\documentclass[12pt, reqno]{amsart}
\usepackage{hyperref}

\usepackage{amssymb,amsmath,graphicx,amsthm}
\usepackage{enumerate}

\def\T{\text}
\newcommand{\Om}{\Omega}

\newcommand{\no}[1]{\|{#1}\|}

\def\C{{\mathbb{C}}}

\numberwithin{equation}{section}

\def\T{\text}

\theoremstyle{plain}

\newtheorem{theorem}{Theorem}[section]

\newtheorem{corollary}[theorem]{Corollary}

\newtheorem{lemma}[theorem]{Lemma}

\newtheorem{proposition}[theorem]{Proposition}

\theoremstyle{definition}

\newtheorem{definition}[theorem]{Definition}

\theoremstyle{remark}

\newtheorem{remark}[theorem]{Remark}

\DeclareMathOperator*{\esssup}{ess\,sup}

\def\B{\mathcal B}

\def\BB{\mathbb B}

\usepackage{amsmath}
\usepackage{amsthm}
\usepackage{amsfonts}
\usepackage{graphics}

\usepackage{amssymb}
\usepackage{verbatim}
\usepackage{amscd}
\usepackage{enumerate}
\usepackage{graphicx}
\numberwithin{equation}{section}

\begin{document}

\title[Bergman-Toeplitz operators]{Bergman-Toeplitz operators between weighted $L^p$-spaces on weakly pseudoconvex domains}

\subjclass[2010]{Primary {47B35}, 32T25; Secondary 32A25, 32A36.}

\keywords{Bergman kernel, Bergman-Toeplitz operator, pseudoconvex domain of finite type, essential norm, Schatten class}

\date{\today}

\author[Tran Vu Khanh, Pham Trong Tien]{Tran Vu Khanh and Pham Trong Tien}

\address{Tran Vu Khanh}
\address{School of Engineering, Tan Tao University, Duc Hoa District, Long An Province, Vietnam}
\email{khanh.tran@ttu.edu.vn}
\address{}
\address{Pham Trong Tien}
\address{Department of Mathematics, Mechanics and Informatics, VNU University of Science, Vietnam National University, Hanoi}%
\address{Thang Long Institute of Mathematics and Applied Sciences, Nghiem Xuan Yem, Hoang Mai, Hanoi, Vietnam}
\email{phamtien@vnu.edu.vn}

\begin{abstract}
In this paper we study the Bergman-Toeplitz operator $T_{\psi}$ induced by $\psi(w) = K_{\Om}^{-\alpha}(w,w)d_{\Om}^{\beta}(w)$ with $\alpha, \beta \geq 0$ acting from a weighted $L^p$-space $L_a^p(\Om)$ to another one $L_a^q(\Om)$ on a large class of pseudoconvex domains of finite type. In the case $1 < p \leq q < \infty$,  the following results are established:
\begin{itemize}
\item[(i)] Necessary and sufficient conditions for boundedness, which generalize the recent results obtained by Khanh, Liu and Thuc.
\item[(ii)] Upper and lower estimates for essential norm, in particular, a criterion for compactness.
\item[(iii)] A characterization of Schatten class membership of this operator on Hilbert space $L^2(\Om)$. 
\end{itemize}
\end{abstract}

\maketitle         
\section{Introduction} \label{S1}
 Let $\Omega$ be a bounded domain in $\mathbb C^n$ with the boundary $\partial\Om$, $d_{\Om}(z)$ the distance from $z$ to $\partial\Om$, and $K_{\Om}$ the Bergman kernel associated to $\Om$. 
For a function $\psi\in L^\infty(\Om)$, the Bergman-Toeplitz operator with symbol $\psi$ is defined by 
$$
T_\psi f (z) := \int_\Om K_{\Om}(z,w)\psi(w)f(w) dV(w),
$$
where $dV(w)$ is the Lebesgue measure on $\Om$. Operators of such type have been intensively studied on (weighted) $L^p$-spaces and Bergman spaces over the unit ball $\mathbb B_n$ in different directions, such as boundedness \cite{TV10}, compactness \cite{S98}, essential norm \cite{MSW13, S07}, and Schatten class membership \cite{APP15, P14, Z07}. It should be noted that the Bergman kernel $K_{\mathbb B_n}$ associated to the unit ball $\mathbb B_n$ has an explicit formula. The case of the absence of an explicit formula for the Bergman kernel $K_{\Om}$ is more complicated and has attracted attention of many researchers.

Recall that when $\psi$ is identically $1$, $T_{\psi}$ reduces to the Bergman projection $P$ that maps from $L^p(\Om)$ to itself with $1 < p < \infty$ on some classes of pseudoconvex domains of finite type such as strongly pseudoconvex domains \cite{PhSt77}, convex domains of finite type \cite{McNSt94}, pseudoconvex domains of finite type in $\C^2$ \cite{NaRoStWa89}.
In order to improve the regularity of the operator $T_{\psi}$ in $L^p$-spaces, \u{C}u\u{c}kovi\'c and McNeal \cite{CuMc06} studied the operator $T_{\psi}$ induced by $\psi(w) = d_{\Om}(w)^{\eta}$ with $\eta > 0$ on strongly pseudoconvex domains. In this case, based on the precise information on the Bergman kernel established by Fefferman \cite{Fef74} on domains of such type, the authors proved that:
\begin{itemize}
\item[(1)] For $0 \leq \eta < n+1$ and $1 < p < \infty$, if $\frac{n+1}{n+1 - \eta} < \frac{p}{p-1}$, then $T_{d_{\Om}^{\eta}}: L^p(\Om) \to L^q(\Om)$ is continuous, where $\frac{1}{q} = \frac{1}{p} - \frac{\eta}{n+1}$; otherwise $T_{d_{\Om}^{\eta}}: L^p(\Om) \to L^q(\Om)$ is continuous for all $p \leq q < \infty$.
\item[(2)] For $\eta \geq n + 1$, $T_{d_{\Om}^{\eta}}: L^1(\Om) \to L^{\infty}(\Om)$ is continuous.
\end{itemize}

Later, Abate, Raissy and Saracco \cite{ARS12} showed that the gain in the exponents in this result is optimal by using geometric characterization of Carleson measures in term of the intrinsic Kobayashi geometry of the domain.
Recently, by choosing $\psi(w) = K^{-\alpha}_{\Om}(w,w)$ with $\alpha \geq 0$, Khanh, Liu and Thuc \cite{KLT18} extended this result to a large class of pseudoconvex domains of finite type whose Bergman kernels have good estimates, called \textit{sharp $\mathcal B$-type} (see, Definition \ref{df-bsharp} below). Moreover, the authors also gave an upper-bound for the norm $\|T_{K^{-\alpha}_{\Om}}\|_{L^p(\Om) \to L^q(\Om)}$. It is worth mentioning that this upper-bound generalized the one for the norm of Bergman projection $B$ on the unit ball in $\C^n$ \cite{Zhu06} and on strongly pseudoconvex domains \cite{Cuc17}.
   
Motivated by some ideas in \cite{KLT18}, in this paper we are interested in the Toeplitz operator $T_{\psi}$ induced by $\psi(w) = K_{\Om}(w,w)^{-\alpha}d_{\Om}(w)^{\beta}$ with $\alpha, \beta \geq 0$,  denoted by $T_{\alpha, \beta}$ for simplicity, acting from a weighted $L^p$-space $L_a^p(\Om)$ to another one $L_a^q(\Om)$. Recall that for $0 < p < \infty$ and $a > -1$, the weighted space $L^p_{a}(\Om)$ consists of all measurable functions $f$ on $\Om$ for which
$$
\|f\|_{p,a} := \left( \int_{\Om} |f(z)|^p dV_{a}(z) \right)^{\frac{1}{p}} < \infty,
$$
where $dV_{a}(z) := d_{\Om}(z)^a dV(z)$.
In the case $a = 0$, we use the symbols $L^p(\Om)$ and $\|\cdot\|_{p}$ instead of $L^p_0(\Om)$ and  $\|\cdot\|_{p,0}$, respectively.

The aim of this paper is not only to study the gain $L^p$-estimate property of the Toeplitz operator $T_{\alpha, \beta}: L^p_a(\Om) \to L^q_a(\Om)$ with $1 < p \leq q < \infty$ as in \cite{KLT18}, but also to investigate its essential norm and membership in the Schatten class.
In details, we firstly generalize all results in \cite{KLT18} for the operator $T_{\alpha, \beta}: L^p_a(\Om) \to L^q_a(\Om)$ with $1 < p \leq q < \infty$ in Section \ref{sec-bd}. Next, Section \ref{sec-ess-nor} is devoted to upper and lower estimates for its essential norm, in particular, a criterion for compactness of this operator. Moreover, a characterization of Schatten class membership of $T_{\alpha, \beta}$ on Hilbert space $L^2(\Om)$ is established in Section \ref{sec-sch}.

Our main results are stated in term of the following quantity. For $1 < p \leq q < \infty$, $a > -1$, and a measurable function $\psi: \Om \to \C$, we put
$$
M_{\Om, p, q, a}^{\psi}(w): = |\psi(w)| K_{\Om}(w,w)^{\frac{1}{p} - \frac{1}{q}}d_{\Om}(w)^{a\left(\frac{1}{q} - \frac{1}{p}\right)}, w \in \Om,
$$
and, as usual, 
$$
\|M_{\Om, p, q, a}^{\psi}\|_{\infty}: = \esssup_{w \in \Om} M_{\Om, p, q, a}^{\psi}(w).
$$
In particular, for the special function $\psi(w): = K_{\Om}(w,w)^{-\alpha}d_{\Om}(w)^{\beta}$ with $\alpha, \beta \geq 0$, we use the symbols $M_{\Om, p, q, a}^{\alpha, \beta}(w)$ and $\|M_{\Om, p, q, a}^{\alpha, \beta}\|_{\infty}$ instead of $M_{\Om, p, q, a}^{\psi}(w)$ and $\|M_{\Om, p, q, a}^{\psi}\|_{\infty}$, respectively. It is easy to see that $M_{\Om, p, q, a}^{\alpha, \beta}(w)$ is continuous on $\Om$, and hence $\|M_{\Om, p, q, a}^{\alpha, \beta}\|_{\infty} = \sup_{w \in \Om} M_{\Om, p, q, a}^{\alpha, \beta}(w)$.

\begin{theorem}\label{thm-main}
Let $1 < p \leq q < \infty$, $-1 < a < \min\{2(p-1), \frac{q}{p'} \}$, $p'$ be the conjugate of $p$, i.e., $\frac1p+\frac1p'=1$,  and $\Om$ be a bounded, pseudoconvex domain in $\C^n$ with smooth boundary. Furthermore, $\Om$ is one of the following domains:
\begin{itemize}
\item[(a)] a strongly pseudoconvex domain;
\item[(b)] a pseudoconvex domain of finite type in $\C^2$;
\item[(c)] a convex domain of finite type;
\item[(d)] a decouple domain of finite type;
\item[(e)] a pseudoconvex domain of finite type whose Levi-form has only one degenerate eigenvalue or comparable eigenvalues.
\end{itemize}
For every $\alpha, \beta \geq 0$, the following statements hold:
\begin{itemize}
\item[(1)] The operator $T_{\alpha, \beta}: L_a^p(\Om) \to L_a^q(\Om)$ is continuous if and only if $M_{\Om, p, q, a}^{\alpha, \beta}(w) \in L^{\infty}(\Om)$.
In this case, 
$$
\|T_{\alpha, \beta}\|_{L_a^p(\Om) \to L_a^q(\Om)} \leq C \left( \dfrac{p'+q}{(1+a)\left(1 - \frac{ap'}{q} \right) }\right)^{\frac{1}{p'} + \frac{1}{q}} \|M_{\Om, p, q, a}^{\alpha, \beta}\|_{\infty},
$$
where $C$ is independent of $p, q, a$.

\item[(2)] If $M_{\Om, p, q, a}^{\alpha, \beta}(w) \in L^{\infty}(\Om)$, then the essential norm of $T_{\alpha, \beta}: L_a^p(\Om) \to L_a^q(\Om)$ satisfies the following estimate
\begin{equation}\label{eq-ess-nor}
\|T_{\alpha, \beta}\|_{e, L_a^p(\Om) \to L_a^q(\Om)} \approx \limsup_{w \to \partial \Om} M_{\Om, p, q, a}^{\alpha, \beta}(w).
\end{equation}
In particular, the operator $T_{\alpha, \beta}: L_a^p(\Om) \to L_a^q(\Om)$ is compact if and only if 
\begin{equation*}
\limsup_{w \to \partial \Om} M_{\Om, p, q, a}^{\alpha, \beta}(w) = 0.
\end{equation*}
\item[(3)] Suppose that the operator $T_{\alpha, \beta}$ is compact on $L^2(\Om)$. For every $s \geq 1$, the operator $T_{\alpha, \beta}$ belongs to Schatten class $S_s$ if and only if $K_{\Om}(w,w)^{1 - s\alpha}d_{\Om}(w)^{s\beta} \in L^1(\Om)$. In the case $s \in (0,1)$, if $ 2\alpha + \beta < 2$ and $K_{\Om}(w,w)^{1 - s\alpha}d_{\Om}(w)^{s\beta} \in L^1(\Om)$, then $T_{\psi} \in S_s$.
\end{itemize}
\end{theorem}

Parts (1) and (2) in Theorem \ref{thm-main} are immediate consequences of the results obtained in Sections \ref{sec-bd} and \ref{sec-ess-nor}, respectively. In details, sufficient conditions for boundedness of the operator $T_{\psi}: L_a^p(\Om) \to L_a^q(\Om)$ and upper estimate for the essential norm $\|T_{\psi}\|_{e, L_a^p(\Om) \to L_a^q(\Om)}$ are established in Theorems \ref{thm:main1} and, respectively, \ref{thm:main1:com} in general case when $\psi$ is a measurable function on $\Om$ and $\Om$ is a bounded pseudoconvex domain whose the Bergman kernel is of sharp $\mathcal B$-type. The necessary ones and lower estimate for $\|T_{\alpha, \beta}\|_{e, L_a^p(\Om) \to L_a^q(\Om)}$ are proved in Theorems \ref{thm:main2} and, respectively, \ref{thm:main2:com} under an additional geometric \textit{$\mathcal B$-polydisc condition} of the domain $\Om$ (see, Definition \ref{df-BP} below). The proof of part (3) in Theorem \ref{thm-main} is stated in Theorem \ref{thm:main:sch}.

Furthermore, in Section \ref{sec-Ber} we also show that all results in Theorem \ref{thm-main} except the estimate \eqref{eq-ess-nor} hold for the operator $T_{\alpha, \beta}$ acting from a weighted Bergman space $A_{a}^p(\Om)$ to another one $A_{a}^q(\Om)$ with $1 < p \leq q < \infty$. 

 {\bf Notations:} Throughout this paper, for every number $p > 1$, we denote by $p'$ its conjugate index, i. e. $\frac{1}{p} + \frac{1}{p'} = 1$.
 We also use the notation $A \lesssim B$ for nonnegative quantities $A$ and $B$ to mean that there is an inessential constant $C> 0$ such that $A \leq C B$; similarly the notation $A \approx B$ means that both $A \lesssim B$ and $B \lesssim A$ hold, where the constant $C$ may change from place to place. 
  The terminology ``\textit{universal positive constant}" means that this constant depends only on the domain $\Om$ (e.g. $n$ and the type of $\Om$).
 
\section{Preliminaries}
In this section we recall the notation of sharp $\mathcal B$-type and an additional geometric hypothesis of $\Om$, under which our main results are established. 

\begin{definition}\label{df-bsharp}
According to \cite[Definitions~2.1 and 2.2]{KLT18}, the Bergman kernel $K_\Om$ associated to a domain $\Om$ is called \textit{of sharp $\B$-type}, if the following conditions hold:
	\begin{enumerate}
		\item[(i)] $K_{\Om}$ is continuous up to the off-diagonal boundary, i.e. $K_\Om\in C((\overline\Om\times\overline\Om) \setminus (\partial \Om \times \partial \Om) )$;
		\item[(ii)] there are universal constants $c$ and $C$ dependent only on $\Om$ such that for every $z\in \overline\Om$ near the boundary $\partial \Om$, we can find a biholomorphism $\Phi_z$ whose holomorphic Jacobian is uniformly nonsingular in the sense that  
		$C^{-1}\le |\det J_\C \Phi_z(w)|\le C$ for all $w$ in a neighborhood of $z$, 
		so that the Bergman kernel $K_{\Om'}$ associated to the domain $\Om':=\Phi_z(\Om)$ is \textit{of sharp $\B$-type at $z':=\Phi_z(z)$}, that is,
		$$
		C^{-1}  \prod_{j=1}^n b_j^2(z',z') \leq  K_{\Om'}(z', z') \leq C \prod_{j=1}^n b_j^2(z',z')
		$$
		 and
		$$
		|K_{\Om'}(z', w')| \leq C \prod_{j=1}^n b_j^2(z',w') \text{ for all } w' \in \Om' \cap \mathbb B(z',c).
	  $$
\end{enumerate}
Here, $\B=\{b_j(z',\cdot)\}_{j=1}^n$  is a \emph{$\B$-system} at $z'$, i. e. there exist a neighbourhood $U$ of $z'$ and a positive integer $m\geq2$ such that for all $w'\in U$, 
  	$$
  	b_{1}(z',w')  := \frac{1}{\delta(z',w')} \quad \T{and} \quad 
  	b_{j}(z',w')   :=  \sum_{k=2}^{m}\left(\dfrac{A_{jk}\left(z'\right)}{\delta(z',w')}\right)^{\frac{1}{k}}, \quad \T{for } j=2,{\scriptstyle\ldots},n,
		$$
  	where $\left\{ A_{jk}:U\rightarrow [0, \infty) \right\}$ are bounded functions such that for each $j$ there exists a $k$ so that $A_{jk}>0$ on $U$; and $\delta(z',w')$ is the pseudo-distance between $z'$ and $w'$, given by
  	\begin{eqnarray}
  	\label{eqn:delta0}
  	\delta(z',w') = d_{\Om'}\left(z'\right) + d_{\Om'}\left(w'\right)+\left|z'_{1}-w'_{1}\right|+\sum_{l=2}^{n}\sum_{s=2}^{m}A_{ls}\left(z'\right)\left|z'_{l}-w'_{l}\right|^{s},
  	\end{eqnarray}
  	under a proper system of coordinates, see \cite{McNSt94}.
\end{definition}

\begin{definition}\label{df-BP}
We say that a bounded smooth pseudoconvex domain $\Om$ in $\C^n$ whose Bergman kernel is of sharp $\B$-type satisfies \textit{$\mathcal B$-polydisc condition}, if there are universal constants $\lambda$ and $C$ such that for every $z \in \Om$ near the boundary $\partial\Om$,
$$
P_{\lambda}(z') \subset \Om':=\Phi_z(\Om), K_{\Om'}(w',w')\le C K_{\Om'}(z',z'), \text{ and } C^{-1} d_{\Om'}(z')  \leq d_{\Om'}(w') \leq C d_{\Om'}(z'),
$$
for all $w'\in P_{\lambda}(z')$, where $\Phi_z$ is the biholomorphism defined in Definition~\ref{df-bsharp}, $z' = \Phi_z(z)$, and 
 	$$
 	P_\lambda(z')=\{w'\in \C^n: |w'_j-z'_j|b_{j}(z',z')\le \lambda, \T{ for all } j= 1,2, \dots, n\}  
 	$$
 	is a $\B$-polydisc with centre $z'$ associated to the $\B$-system in Definition~\ref{df-bsharp}.
\end{definition}

\begin{remark}
In \cite[Theorem~4.1]{KLT18} using the results of Fefferman \cite{Fef74}, Catlin \cite{Cat89}, McNeal \cite{McN91, McN94}, McNeal and Stein \cite{McNSt94}, and Cho \cite{Cho96, Cho02}, Khanh, Liu and Thuc  proved that the Bergman kernels associated to all domains $\Om$ in Theorem \ref{thm-main} are of sharp $\mathcal B$-type and these domains satisfy $\mathcal B$-polydisc condition. Furthermore, by the conditions of $A_{jk}$ in Definition \ref{df-bsharp}.(ii) for the domains in Theorem \ref{thm-main}, the inequality
\begin{equation}\label{eq-ine-dk}
d_{\Om}(z)^{-2} \leq K_{\Om}(z,z)
\end{equation}
holds for all $z \in \Om$.
\end{remark}

Next, we give weighted $L^p$-estimates for the Bergman kernel $K_{\Om}(\cdot, z)$. To do this, we recall the following auxiliary result, which is proved in \cite[Proposition~2.4]{KLT18} and plays an important role in this paper.
 
 \begin{proposition}\label{prop:main} Let $\Om$ be a domain in $\C^n$ such that the Bergman kernel $K_{\Om}$ is of sharp $\mathcal{B}$-type. Then, for each $z_0\in \partial\Om$, there is a neighbourhood $U$ of $z_0$ such that  for any $a \ge 1$ and $-1 < b <2a-2$, 
	\begin{equation}
	\label{eqn:Iab}
	\begin{split}
	I_{a,b}\left(z\right) &:= \intop_{\Omega\cap U}\left|K_{\Om}\left(z,w\right)\right|^{a} d_{\Om}\left(w\right)^{b} dV(w) \\
		&\leq C\frac{2a-1}{(2a-2-b)(b + 1)}K_{\Om}(z,z)^{a-1}d_{\Om}\left(z\right)^{b}
	\end{split}
	\end{equation}
	for every $z\in\Omega\cap U$ and some constant $C$ dependent only on $U$ and $\Om$.
\end{proposition}

Using this proposition, we can get the upper estimate for the norm of $K_{\Om}(\cdot,z)$ in $L_a^p(\Om)$.

\begin{lemma}\label{lem-nor-ue}
Let $ 1 \leq p < \infty$, $-1 < a < 2(p - 1)$, and $\Om$ be a bounded domain in $\C^n$ such that the Bergman kernel $K_{\Om}$ is of sharp $\mathcal{B}$-type.
 Then for every $z \in \Om$,
$$
\|K_{\Om}(\cdot,z)\|_{p, a} \leq C K_{\Om}(z,z)^{1 - \frac{1}{p}}d_{\Om}(z)^{\frac{a}{p}},
$$
for some constant $C$ independent of $z$.
\end{lemma}
\begin{proof}
We choose a covering $\{U_j\}_{j=0}^N$ to $\overline\Om$ so that $U_{0}\Subset \Omega$,  $\partial\Omega\subset\bigcup_{j=1}^{N}U_{j}$,
and the integral estimates in Proposition~\ref{prop:main} hold on $U_j$ with some constant $C_j$ for all $j=1,\dots, N$.

Since $K_\Om\in C((\overline\Om\times\overline\Om) \setminus (\partial \Om \times \partial \Om) )$, there is a constant $C > 0$ such that
$$
\left|K_{\Om}\left(w,z\right)\right|\le C
\quad \T{
for all }\quad (w,z)\in \left(\bigcup_{j=1}^N \left( \left( \overline{\Om} \cap \overline{U_j} \right) \times (\overline{\Om}\setminus U_j)\right)\bigcup \left(\overline{U_0} \times \overline{\Om}\right)\right).
$$
Using this and Proposition \ref{prop:main} for $U_j, j = 1, \dots, N$, we get that  for every $z \in \Om$,
$$
\int_{\Om \cap U_j} |K_{\Om}(w,z)|^p d_{\Om}(w)^a dV(w) \leq
C^p \|1\|_{p,a}^p, \text{ if } z \in \Om \setminus U_j,
$$
and
$$
\int_{\Om \cap U_j} |K_{\Om}(w,z)|^p d_{\Om}(w)^a dV(w) \leq
C_j \dfrac{2p-1}{(2p-2-a)(a+1)} K_{\Om}(z,z)^{p-1}d_{\Om}(z)^a, \text{ if } z \in \Om \cap U_j.
$$
Hence, for every $z \in \Om$,
\begin{align*}
\|K_{\Om}(\cdot,z)\|_{p, a}^p &= \int_{\Om} |K_{\Om}(w,z)|^p d_{\Om}(w)^a dV(w)\\
& \leq \int_{U_0} |K_{\Om}(w,z)|^p d_{\Om}(w)^a dV(w) + \sum_{j = 1}^N \int_{\Om \cap U_j} |K_{\Om}(w,z)|^p d_{\Om}(w)^a dV(w) \\
& \leq C^p \|1\|_{p,a}^p + \sum_{j = 1}^N \max \left\{C^p \|1\|_{p,a}^p,  C_j \dfrac{2p-1}{(2p-2-a)(a+1)} K_{\Om}(z,z)^{p-1}d_{\Om}(z)^a \right\}. 
\end{align*}
Moreover, since $a < 2(p-1)$ and, by \eqref{eq-ine-dk}, $d_{\Om}(z)^{-2} \leq K_{\Om}(z,z), z \in \Om$,
$$
K_{\Om}(z,z)^{1-\frac{1}{p}}d_{\Om}(z)^{\frac{a}{p}} \geq d_{\Om}(z)^{\frac{a+2}{p} - 2} \to \infty \text{ as } z \to \partial \Om.
$$
From this and the above inequality the desired estimate follows.
\end{proof}

The lower estimate for $\|K_{\Om}(\cdot, z)\|_{p,a}$ is established under $\mathcal B$-polydisc condition.

\begin{lemma}\label{lem-nor-le}
Let $\Om$ be a bounded smooth pseudoconvex domain in $\C^n$ such that the Bergman kernel is of sharp $\B$-type and $\mathcal B$-polydisc condition is satisfied. For every $p \geq 1$, $a > -1$, and $z \in \Om$ near the boundary $\partial \Om$,
$$
\|K_{\Om}(\cdot, z)\|_{p, a} \gtrsim  K_{\Om}(z,z)^{1 - \frac{1}{p}}d_{\Om}(z)^{\frac{a}{p}}.
$$
\end{lemma}
\begin{proof}
For every $z \in \Om$ near the boundary $\partial \Om$, put $w' = \Phi_z(w)$ and $\Om' = \Phi_z(\Om)$. By the invariant formula, 
$$
K_{\Om}(z,w) = \det J_\C\Phi_z(z) K_{\Phi_z(\Om)}(\Phi_z(z), \Phi_z(w)) \overline{\det J_\C\Phi_z(w)},
$$
the fact that $ C^{-1} \leq |\det J_\C\Phi_z(w)|\leq C$ for all $w$ in a neighborhood  of $z$, we get
\begin{align*}
\|K_{\Om}(\cdot, z)\|_{p, a}^p & = \int_\Om |K_{\Om}(w,z)|^p d_{\Om}(w)^{a}dV(w) \\
 & = \int_{\Om'} |\det J_{\C}\Phi_z^{-1}(z')|^{-p}  |K_{\Om'}(w',z')|^p |\det J_{\C}\Phi_z^{-1}(w')|^{-p} \\
& \qquad \qquad \qquad \times  d_{\Om}(\Phi_z^{-1}(w'))^{a} |\det J_{\C}\Phi_z^{-1}(w')| dV(w') \\
 \ge & \int_{P_{\lambda}(z')}  |K_{\Om'}(w',z')|^p d_{\Om}(\Phi_z^{-1}(w'))^{a}  |\det J_{\C}\Phi_z(z)|^p |\det J_{\C}\Phi_z^{-1}(w')|^{1-p} dV(w') \\
 \gtrsim &\;  \int_{P_{\lambda}(z')} |K_{\Om'}(w',z')|^p d_{\Om'}(w')^{a}dV(w') \gtrsim \;  d_{\Om'}(z')^{a} \int_{P_\lambda(z')} |K_{\Om'}(w',z')|^pdV(w')\\
\gtrsim &\; d_{\Om'}(z')^{a}K_{\Om'}(z',z')^{p}\T{Vol}(P_\lambda(z')) \gtrsim \;  K_{\Om'}(z',z')^{p - 1}d_{\Om'}(z')^{a},
\end{align*}
where we use the sub-mean property and the fact that
$$
\T{Vol}(P_\lambda(z'))=\pi^n\lambda^{2n}\left(\prod_{j=1}^n b_{j}(z',z')\right)^{-2}\approx (K_{\Om'}(z',z'))^{-1}.
$$	
Moreover, since $K_{\Om}(z,z)=|\det J_\C\Phi_z(z)|^2K_{\Om'}(z',z')$ and $C^{-1} \leq |\det J_\C\Phi_z(z)|\leq C$,   
\begin{eqnarray*}
\|K_{\Om}(\cdot, z)\|_{p, a} \gtrsim  K_{\Om}(z,z)^{1-\frac{1}{p}}d_{\Om}(z)^{\frac{a}{p}}.
\end{eqnarray*}
\end{proof}

\section{Boundedness}\label{sec-bd}

In this section we give sufficient conditions for boundedness of the general operator $T_{\psi}: L_a^p(\Om) \to L_a^q(\Om)$ and necessary ones for the special operator $T_{\alpha, \beta}: L_a^p(\Om) \to L_a^q(\Om)$ with $1 < p \leq q < \infty$.

\subsection{Sufficiency}
Using Proposition \ref{prop:main} and some ideas in the proof of the generalized version of Schur's test \cite[Theorem~5.1]{KLT18}, we get the following result.

\begin{proposition}\label{prop:U} 
Let $1 < p \le q <\infty$, $-1 < a < \frac{q}{p'}$, and $\Om$ be a bounded domain in $\C^n$ such that the Bergman kernel $K_{\Om}$ is  of sharp $\B$-type. Suppose that $\psi:\Om\to \C$ is a measurable function such that $M^{\psi}_{\Om, p, q, a}(w) \in L^{\infty}(\Om)$.
Then, for each $z_0\in \partial\Om$, there exists a neighbourhood $U$ of $z_0$ such that the Toeplitz operator $T_{\psi,U}$ defined by
$$
(T_{\psi,U}f)(z):=\int_{\Om\cap U}K_{\Om}(z,w)\psi(w)f(w)dV(w)\quad \T{for }z\in \Om\cap U, 
$$ 
maps from $L^p_{a}(\Om\cap U)$ to $L^q_{a}(\Om\cap U)$ continuously and 
\begin{eqnarray}
\label{eqn:norm T local}
\no{T_{\psi,U}}_{L^p_{a}(\Om\cap U) \to L^q_{a}(\Om\cap U)}\le C\left(\dfrac{p'+q}{(1+a)(1 - \frac{a p'}{q})}\right)^{\frac{1}{p'}+\frac{1}{q}} \|M^{\psi}_{\Om, p, q, a}\|_{\infty}, 
\end{eqnarray}
where $C$ is independent of $p,q$ and $a$, and
$$
L^p_{a}(\Om\cap U) = \left \{f \text{ is measurable on } \Om \cap U:  \no{f}_{L^p_{a}(\Om\cap U)} = \left( \int_{\Om \cap U} |f(z)|^p dV_{a}(z) \right)^{\frac{1}{p}} < \infty \right\}.
$$
\end{proposition}
\begin{proof}
Take the neighbourhood $U$ of $z_0$ as in Proposition \ref{prop:main}. 
We put $\delta := \frac{1}{p'}$ and
$$
g(z) := d_{\Om}(z)^{-\gamma}, h_1(w) := d_{\Om}(w)^{-\gamma},
h_2(w) := K_{\Om}(w,w)^{\frac{1}{p} - \frac{1}{q}}d_{\Om}(w)^{\frac{a}{q} - \gamma},
$$
with
\begin{equation}\label{eq361}
\max \left\{ 0, \frac{a}{q} \right\} 
< \gamma < \min \left\{ \dfrac{1}{p'}, \dfrac{a + 1}{q} \right\}.
\end{equation}

Since $M^{\psi}_{\Om, p, q, a}(w) \in L^{\infty}(\Om)$,
\begin{align}\label{eq-infty}
\nonumber & \esssup_{w \in \Om}|h_1^{-1}(w)h_2(w)\psi(w) d_{\Om}(w)^{-\frac{a}{p}}| \\
\nonumber =  &\esssup_{w \in \Om} |\psi(w)| d_{\Om}(w)^{\gamma} K_{\Om}(w,w)^{\frac{1}{p} - \frac{1}{q}} d_{\Om}(w)^{\frac{a}{q}  - \gamma} d_{\Om}(w)^{-\frac{a}{p}}\\
 = & \esssup_{w \in \Om} |\psi(w)| K_{\Om}(w,w)^{\frac{1}{p} - \frac{1}{q}}d_{\Om}(w)^{a\left(\frac{1}{q} - \frac{1}{p} \right)} = \|M^{\psi}_{\Om, p, q, a}\|_{\infty} < \infty.
\end{align}

Put $a_1:= \delta p' = 1$ and $ b_1: =  - \gamma p' \in (-1, 2a_1 - 2)$ by \eqref{eq361}. Hence,  using Proposition \ref{prop:main} for $(a_1, b_1)$,
we obtain that, for every $z \in \Om \cap U$,
\begin{align}\label{eq-g}
\nonumber \int_{\Om \cap U} |K_{\Om}(z,w)|^{\delta p'} h_1(w)^{p'}dV(w) = & \int_{\Om \cap U} |K_{\Om}(z,w)|^{\delta p'} d_{\Om}(w)^{ - \gamma p'} dV(w) \\
 \leq & \; C_1 \dfrac{1}{\gamma p'(1 - \gamma p')} d_{\Om}(z)^{ - \gamma p'} =  \; C_1 \tau_1(\gamma) g(z)^{p'},
\end{align}
where $\tau_1(\gamma): = \dfrac{1}{\gamma p'(1 - \gamma p')}$.

On the other hand, it is clear that $a_2:= (1 - \delta) q = \frac{q}{p} \geq 1, b_2: = a - \gamma q  \in (-1, 2a_2 - 2)$ by \eqref{eq361}. Then, using Proposition \ref{prop:main} for $(a_2, b_2)$, we get that, for every $w \in \Om \cap U$,
\begin{align}\label{eq-h2}
\nonumber & \int_{\Om \cap U} |K_{\Om}(z,w)|^{(1-\delta)q} g(z)^q d_{\Om}(z)^a dV(z) =  \int_{\Om \cap U} |K_{\Om}(z,w)|^{(1-\delta)q} d_{\Om}(z)^{a - \gamma q}dV(z) \\
\leq &\; C_2 \dfrac{\frac{2q}{p} - 1}{\left(\frac{2q}{p} - 2 - a + \gamma q \right)\left(a  - \gamma q + 1\right)}K_{\Om}(w,w)^{\frac{q}{p} - 1}d_{\Om}(w)^{a - \gamma q} =  \; C_2 \tau_2(\gamma) h_2(w)^q,
\end{align}
where $\tau_2(\gamma) = \dfrac{\frac{2q}{p} - 1}{\left(\frac{2q}{p} - 2 - a + \gamma q \right)\left(a  - \gamma q + 1\right)}$.

Now, using H\"older's inequality and \eqref{eq-g}, for every $f \in L^p_a(\Om \cap U)$ and every $z \in \Om \cap U$, we obtain
\begin{align*}
\left| T_{\psi, U}f(z) \right| & \leq \int_{\Om \cap U} \left( |K_{\Om}(z,w)|^{\delta}h_1(w) \right) \left( |K_{\Om}(z,w)|^{1 - \delta} h_1(w)^{-1} |\psi(w)| |f(w)| \right)dV(w) \\
& \leq \left( C_1 \tau_1(\gamma) g(z)^{p'} \right)^{\frac{1}{p'}} \left( \int_{\Om \cap U}|K_{\Om}(z,w)|^{(1 - \delta)p} h_1(w)^{-p}  |\psi(w)|^p |f(w)|^p  dV(w)  \right)^{\frac{1}{p}} \\
& \leq \left( C_1 \tau_1(\gamma) \right)^{\frac{1}{p'}} \left( \int_{\Om \cap U}|K_{\Om}(z,w)|^{(1 - \delta)p} g(z)^p h_1(w)^{-p}  |\psi(w)|^p |f(w)|^p  dV(w)  \right)^{\frac{1}{p}}.
\end{align*}
From this, \eqref{eq-infty}, \eqref{eq-h2}, and using Minkowski's inequality (see, \cite[Theorem~5.3]{KLT18}) for $\eta = \frac{q}{p} \geq 1$, we get that for every $f \in L^p_a(\Om \cap U)$,
\begin{align*}
& \|T_{\psi, U} f\|_{L_a^q(\Om \cap U)}^p  \leq   \left( C_1 \tau_1(\gamma) \right)^{\frac{p}{p'}} \\
& \times \left( \int_{\Om \cap U} \left( \int_{\Om \cap U}|K_{\Om}(z,w)|^{(1 - \delta)p} g(z)^p h_1(w)^{-p}  |\psi(w)|^p |f(w)|^p dV(w) \right)^{\frac{q}{p}} dV_a(z) \right)^{\frac{p}{q}}\\
=  &  \left( C_1 \tau_1(\gamma) \right)^{\frac{p}{p'}} \\
 & \times \left( \int_{\Om \cap U} \left( \int_{\Om \cap U}|K_{\Om}(z,w)|^{(1 - \delta)p} g(z)^p d_{\Om}(z)^{\frac{ap}{q}} h_1(w)^{-p}  |\psi(w)|^p |f(w)|^p dV(w) \right)^{\frac{q}{p}} dV(z) \right)^{\frac{p}{q}} \\
 \leq & \left( C_1 \tau_1(\gamma) \right)^{\frac{p}{p'}} \\
 & \times \int_{\Om \cap U} \left( \int_{\Om \cap U} |K_{\Om}(z,w)|^{(1 - \delta)q} g(z)^q d_{\Om}(z)^{a} h_1(w)^{-q}  |\psi(w)|^q |f(w)|^q dV(z) \right)^{\frac{p}{q}}dV(w) \\
 \leq & \left( C_1 \tau_1(\gamma) \right)^{\frac{p}{p'}} \left(C_2 \tau_2(\gamma) \right)^{\frac{p}{q}} \int_{\Om \cap U} h_2(w)^p h_1(w)^{-p}  |\psi(w)|^p d_{\Om}(w)^{-a} |f(w)|^p dV_a(w) \\
 \leq & \left( C_1 \tau_1(\gamma) \right)^{\frac{p}{p'}} \left(C_2 \tau_2(\gamma) \right)^{\frac{p}{q}} \|M^{\psi}_{\Om, p, q, a}\|_{\infty}^p \|f\|^p_{L_a^p(\Om \cap U)}
\end{align*}

Consequently, the Toeplitz operator $T_{\psi, U}: L^p_{a}(\Om\cap U) \to L^q_{a}(\Om\cap U)$ is continuous and  
$$
\no{T_{\psi, U}}_{L_a^p(\Om \cap U) \to L^q_{a}(\Om\cap U)} \le C \tau_1(\gamma)^{\frac{1}{p'}} \tau_2(\gamma)^{\frac{1}{q}}  \|M^{\psi}_{\Om, p, q, a}\|_{\infty}.
$$
Now we give an upper estimate for $\tau_1(\gamma)^{\frac{1}{p'}} \tau_2(\gamma)^{\frac{1}{q}}$. By \eqref{eq361}, $0 < \gamma q - a < 1$. Then by the inequality $\frac{x+a}{x+b}\le \frac{a}{b}$ for $a\ge b>0$ and $x\ge 0$, we obtain
\begin{align*}
\tau_1(\gamma)^{\frac{1}{p'}} \tau_2(\gamma)^{\frac{1}{q}} & =  \left(\dfrac{1}{ \gamma p' (1 - \gamma p')}\right)^{\frac{1}{p'}} \left( \dfrac{\left( \frac{2q}{p} - 2 \right) + 1}{\left(\frac{2q}{p} - 2 - a + \gamma q \right)\left(a  - \gamma q + 1\right)}\right)^{\frac{1}{q}} \\
& \leq \left(\dfrac{1}{ \gamma p' (1 - \gamma p')}\right)^{\frac{1}{p'}} \left( \dfrac{1}{\left(\gamma q - a \right)\left(a  - \gamma q + 1\right)}\right)^{\frac{1}{q}}.
\end{align*}
It is easy to check that the number $\gamma_0:= \frac{1 + a}{p'+q}$ satisfies \eqref{eq361}. Then 
\begin{align*}
\tau_1(\gamma_0)^{\frac{1}{p'}} \tau_2(\gamma_0)^{\frac{1}{q}} & \leq \left( \dfrac{1}{\frac{p'(1+a)}{p'+q} \frac{q - a p'}{p'+q}} \right)^{\frac{1}{p'} + \frac{1}{q}} = \left( \dfrac{(p'+q)^2}{p'(1+a) (q - a p')} \right)^{\frac{1}{p'} + \frac{1}{q}} \\
& = \left( \frac{1}{p'} + \frac{1}{q} \right)^{\frac{1}{p'} + \frac{1}{q}} \left(  \dfrac{p'+q}{(1+a) (1 - \frac{a p'}{q})} \right)^{\frac{1}{p'} + \frac{1}{q}} \\
& \leq 4 \left(  \dfrac{p'+q}{(1+a) (1 - \frac{a p'}{q})} \right)^{\frac{1}{p'} + \frac{1}{q}},
\end{align*}
where the last inequality follows by $x^x\le 4$ for $x=\frac{1}{p'}+\frac{1}{q}\in [0,2]$. Thus, the desired estimate \eqref{eqn:norm T local} follows.
\end{proof}

\begin{theorem}\label{thm:main1}
	Let $1 < p \leq q < \infty$, $-1 < a < \frac{q}{p'}$, and $\Om$ be a bounded domain in $\C^n$ such that the Bergman kernel  $K_{\Om}$ is of sharp $\B$-type. 
	Suppose that  $\psi:\Om\to \C$ is a measurable function such that $M^{\psi}_{\Om, p, q, a}(w) \in L^{\infty}(\Om)$. Then the Toeplitz operator $T_{\psi}:L^p_{a}(\Om)\to L^q_{a}(\Om)$ is continuous. Furthermore, 
		$$
		\no{T_\psi}_{L^p_{a}(\Om)\to L^q_{a}(\Om)} \le C \left(\dfrac{p'+q}{(1+a)(1 - \frac{a p'}{q})}\right)^{\frac{1}{p'}+\frac{1}{q}}  \|M^{\psi}_{\Om, p, q, a}\|_{\infty}, 
		$$
		where $C$ is independent of $p, q, a$.
\end{theorem}
\begin{proof}
We choose a partition of unity $\{\chi_j\}_{j=0}^{N}$ and a covering $\{U_j\}_{j=0}^N$ to $\overline\Om$ 
so that ${\T{supp}(\chi_j)}\Subset U_j$, $U_{0}\Subset \Omega$,  $\partial\Omega\subset\bigcup_{j=1}^{N}U_{j}$,
and the results in Proposition~\ref{prop:U} hold on $U_j$ for all $j=1,\dots, N$. Then, for every $f \in L^p_a(\Om)$, we can decompose $T_\psi f$ as 
$$
T_{\psi}f=\sum_{j=0}^N\chi_jT_{\psi}f=\chi_0T_\psi f+\sum_{j=1}^N \chi_j T_{\psi}(f{\bf 1}_{\Om\cap U_j})+\sum_{j=1}^N \chi_j T_{\psi}(f{\bf 1}_{\Om\setminus U_j}),
$$
where ${\bf 1}_A$ is the characteristic function of a subset $A \subset \Om$. From this it follows that
\begin{eqnarray*}
\label{eqn:a1}\no{T_{\psi}f}_{q, a}\le \no{\chi_0T_\psi f}_{q, a}+\sum_{j=1}^N \no{\chi_j T_{\psi}(f{\bf 1}_{\Om\cap U_j})}_{q, a}+\sum_{j=1}^N \no{\chi_j T_{\psi}(f{\bf 1}_{\Om\setminus U_j})}_{q, a}.
\end{eqnarray*}
To continue, we need several estimates for the norms $\no{\chi_0T_\psi f}_{q, a}$, $\no{\chi_j T_{\psi}(f{\bf 1}_{\Om\setminus U_j})}_{q, a}$, and $\no{\chi_j T_{\psi}(f{\bf 1}_{\Om\cap U_j})}_{q, a}$ with $j = 1,\dots, N$.

\textbf{Estimates for $\no{\chi_j T_{\psi}(f{\bf 1}_{\Om\setminus U_j})}_{q,a}$ and $\no{\chi_0T_\psi f}_{q, a}$}.  Since 
$$
K_{\Om}\in C\left(\left(\overline{\Omega}\times\overline{\Omega}\right)\setminus \left( \partial \Om \times \partial\Omega\right) \right),
$$ 
there exists a positive constant $C$ such that 
\[
\left|K_{\Om}\left(z,w\right)\right|\le C
\quad \T{
for all }\quad (z,w)\in \left(\bigcup_{j=1}^N \left( \T{supp}(\chi_j) \times (\overline{\Om}\setminus U_j)\right)\bigcup \left(\T{supp}(\chi_0)\times \overline{\Om}\right)\right).
\] 
Thus, for $j=1,\dots,N$, and $z\in \Om$, using H\"older's inequality and \eqref{eq-ine-dk}, we get
\begin{align*}
&\left|\left(\chi_jT_{\psi} (f{\bf1}_{\Om\setminus U_j})\right)(z)\right| \\
= & \left|\int_{\Om}\chi_j(z)K_{\Om}(z,w)\psi(w)f(w){\bf1}_{\Om\setminus U_j}(w)dV(w)\right| \\
\leq &\; C \|M^{\psi}_{\Om, p, q, a}\|_{\infty} \int_{\Om} |f(w)| K_{\Om}(w,w)^{\frac{1}{q}-\frac{1}{p}}d_{\Om}(w)^{-a\left(\frac{1}{q} - \frac{1}{p}\right)} dV(w) \\
= &\; C \|M^{\psi}_{\Om, p, q, a}\|_{\infty} \int_{\Om} |f(w)| d_{\Om}(w)^{\frac{a}{p}} K_{\Om}(w,w)^{\frac{1}{q}-\frac{1}{p}}d_{\Om}(w)^{-\frac{a}{q}} dV(w) \\
\leq & \; C \|M^{\psi}_{\Om, p, q, a}\|_{\infty} \left( \int_{\Om}|f(w)|^p d_{\Om}(w)^{a}dV(w) \right)^{\frac{1}{p}} \\
& \qquad \qquad  \qquad \qquad \qquad \qquad \times  \left( \int_{\Om} K_{\Om}(w,w)^{p'\left(\frac{1}{q}-\frac{1}{p}\right)}d_{\Om}(w)^{- \frac{ap'}{q}}dV(w)\right)^{\frac{1}{p'}} \\
\leq & \; C \|M^{\psi}_{\Om, p, q, a}\|_{\infty} \left( \int_{\Om}|f(w)|^p d_{\Om}(w)^{a}dV(w) \right)^{\frac{1}{p}} \left( \int_{\Om} d_{\Om}(w)^{p'\left(\frac{2}{p}-\frac{a + 2}{q} \right)}dV(w)\right)^{\frac{1}{p'}} \\
\leq & \; C \|M^{\psi}_{\Om, p, q, a}\|_{\infty} \|1\|_{p',b} \|f\|_{p,a},
\end{align*}
where, obviously, $b := p'\left(\frac{2}{p}-\frac{a+2}{q} \right) > -1$ by hypothesis, and hence, $\|1\|_{p',b} < \infty$. Thus,
\begin{eqnarray*}
\label{eqn:a2}\no{\chi_jT_{\psi} (f{\bf1}_{\Om\setminus U_j})}_{q, a}\le \; C \|M^{\psi}_{\Om, p, q, a}\|_{\infty} \|1\|_{q, a} \|1\|_{p',b} \no{f}_{p, a},
\end{eqnarray*}
where $C$ is independent of $p, q, a$. Analogously, 
\begin{eqnarray*}
\label{eqn:a3}\no{\chi_0T_{\psi} (f)}_{q,a}\le C \|M^{\psi}_{\Om, p, q, a}\|_{\infty}  \|1\|_{q, a} \|1\|_{p',b} \no{f}_{p,a}.
\end{eqnarray*}

\textbf{Estimates for $\no{\chi_j T_{\psi}(f{\bf 1}_{\Om\cap U_j})}_{q, a}$}. For $j=1,\dots,N$, we have
\begin{align*}
\no{\chi_jT_{\psi} (f{\bf1}_{\Om \cap U_j})}_{q, a} & = \left(\int_{\Om} \left|\chi_j(z) T_{\psi}(f {\bf 1}_{\Om \cap U_j})(z)\right|^qdV_{a}(z)\right)^\frac{1}{q} \\
& \leq \left(\int_{\Om \cap U_j} \left|T_{\psi}(f {\bf 1}_{\Om \cap U_j})(z)\right|^qdV_{a}(z)\right)^\frac{1}{q} \\
& =  \left(\int_{\Om \cap U_j} \left|T_{\psi, U_j}(f{\bf 1}_{\Om\cap U_j})(z)\right|^qdV_{a}(z)\right)^\frac{1}{q} \\
& = \|T_{\psi, U_j}(f{\bf 1}_{\Om\cap U_j})\|_{L^q_a(\Om \cap U_j)}. \\
\end{align*}
From this and Proposition~\ref{prop:U}, it follows that
\begin{eqnarray*}
\label{eqn:a4}\no{\chi_jT_{\psi} (f{\bf 1}_{\Om\cap U_j})}_{q, a}\le & C_j \left(\dfrac{p'+q}{(1+a)(1 - \frac{ap'}{q})}\right)^{\frac{1}{p'}+\frac{1}{q}} \|M^{\psi}_{\Om, p, q, a}\|_{\infty} \|f{\bf 1}_{\Om\cap U_j}\|_{L^p_a(\Om \cap U_j)},
\end{eqnarray*}
 for every  $j=1,\dots,N$, where $C_j$ are independent of $p, q, a$.

From the above estimates, we get 
$$
\no{T_{\psi} f}_{q, a}\le C \left[ \|1\|_{q, a} \|1\|_{p', b} + \left(\dfrac{p'+q}{(1+a)(1 - \frac{a p'}{q})}\right)^{\frac{1}{p'}+\frac{1}{q}} \right] \|M^{\psi}_{\Om, p, q, a}\|_{\infty} \no{f}_{p,a},
$$
for every $f\in L^p_{a}(\Om)$, where $C$ is independent of $p, q, a$.

On the other hand, it is not difficult to check that 
$$
\|1\|_{q,a} \leq \left( \dfrac{C}{1+a} \right)^{\frac{1}{q}} \text{ and }  \|1\|_{p',b} \leq \left( \dfrac{C}{1+b} \right)^{\frac{1}{p'}} \leq \left( \dfrac{C}{1 - \frac{ap'}{q}} \right)^{\frac{1}{p'}}
$$
for some constant $C$ independent of $p, q, a$. From this and the hypothesis on $p, q, a$, it follows that
\begin{align*}
\|1\|_{q, a} \|1\|_{p', b} & \leq \left(\dfrac{C}{(1+a)(1 - \frac{a p'}{q})}\right)^{\frac{1}{p'}+\frac{1}{q}} (1+a)^{\frac{1}{p'}}\left(1 - \frac{ap'}{q}\right)^{\frac{1}{q}} \\
& \leq C^{\frac{1}{p'}+\frac{1}{q}} \left(\dfrac{p'+q}{(1+a)(1 - \frac{a p'}{q})}\right)^{\frac{1}{p'}+\frac{1}{q}}.
\end{align*} 
Consequently,
$$
\no{T_{\psi}}_{L^p_a(\Om) \to L^q_a(\Om)} \le C \left(\dfrac{p'+q}{(1+a)(1 - \frac{a p'}{q})}\right)^{\frac{1}{p'}+\frac{1}{q}} \|M^{\psi}_{\Om, p, q, a}\|_{\infty}
$$
for some constant $C$ independent of $p, q, a$.
\end{proof}

The following result follows immediately from Theorem \ref{thm:main1} which is a generalization of \cite[Theorem~2.7(1)]{KLT18}.

\begin{corollary}\label{cor-1}
Let $1 < p \leq q < \infty$ and $\Om$ be a bounded domain in $\C^n$ such that the Bergman kernel $K_{\Om}$ is of sharp $\B$-type. Suppose that  $\psi:\Om\to \C$ is a measurable function such that $M^{\psi}_{\Om, p, q, 0}(w) \in L^{\infty}(\Om)$, i.e.
\begin{equation*}
\esssup_{w \in \Om}\left|\psi(w)\right| K_{\Om}(w,w)^{\frac{1}{p}-\frac{1}{q}} < \infty.
\end{equation*} 
Then the Toeplitz operator $T_{\psi}:L^p(\Om)\to L^q(\Om)$ is continuous. Furthermore, 
		$$
		\no{T_\psi}_{L^p(\Om)\to L^q(\Om)} \le C   \left(\dfrac{p}{p-1} + q\right)^{1 - \frac{1}{p}+\frac{1}{q}}\|M^{\psi}_{\Om, p, q, 0}\|_{\infty}, 
		$$
		where $C$ is dependent only on $\Om$. 		
\end{corollary}

\subsection{Necessity} \label{S3}
The necessary conditions for boundedness of the Toeplitz operator $T_{\alpha, \beta}: L_a^p(\Om) \to L_a^q(\Om)$ is given under $\mathcal B$-polydisc condition. 

  \begin{theorem}\label{thm:main2}
 	Let $1 < p \leq q < \infty$, $-1 < a < \min\{2(p - 1), q - 1\}$, and $\Om$ be a bounded smooth pseudoconvex domain in $\C^n$ such that the Bergman kernel is of sharp $\B$-type and $\mathcal B$-polydisc condition is satisfied. 
If the operator $T_{\alpha, \beta}: L^p_{a}(\Om) \to L^q_{a}(\Om)$ is continuous, then $M^{\alpha, \beta}_{\Om, p, q, a}(w) \in L^{\infty}(\Om)$, i.e.
$$
\sup_{w \in \Om} K_{\Om}(w,w)^{-\alpha + \left(\frac{1}{p} - \frac{1}{q}\right)}d_{\Om}(w)^{\beta + a\left(\frac{1}{q} - \frac{1}{p}\right)} < \infty.
$$
 \end{theorem}
 \begin{proof} Similarly to \cite[Theorem~3.1]{KLT18}, the proof is based on upper and lower estimates of 
\begin{equation} \label{eq-sK}
\int_\Om |K_{\Om}(w,z)|^2K_{\Om}(w,w)^{-\alpha}d_{\Om}(w)^{\beta}dV(w), z \in \Om.
\end{equation}
 	
\textbf{Upper estimate.} Since  $K_{\Om}(w,z)$ is holomorphic in $w\in \Om$, 
 	$$
 	K_{\Om}(w,z)=P(K_{\Om}(\cdot,z))(w)=\int_{\Om}K_{\Om}(w,\xi)K_{\Om}(\xi,z)dV(\xi),
 	$$
   hence 
 	$$
 	\overline{K_{\Om}(w,z)}=\overline{\int_{\Om}K_{\Om}(w,\xi)K_{\Om}(\xi,z)dV(\xi)}=\int_{\Om}K_{\Om}(\xi,w)K_{\Om}(z,\xi)dV(\xi).
 	$$  
Using this and Fubini's theorem, we get
 		\begin{align}\label{eqn:Ksquare1}
 		\nonumber &\intop_{\Omega}\left|K_{\Om}\left(w,z\right)\right|^{2}K_{\Om}(w,w)^{-\alpha}d_{\Om}(w)^{\beta}dV(w) \\
		\nonumber = & \intop_{\Omega}K_{\Om}(w,z)K_{\Om}(w,w)^{-\alpha}d_{\Om}(w)^{\beta}\left(\int_{\Om}K_{\Om}(\xi,w)K_{\Om}(z,\xi)dV(\xi)\right)dV(w)\\
 		\nonumber = & \intop_{\Om}\left(\intop_{\Om}K_{\Om}(\xi,w)K_{\Om}(w,w)^{-\alpha}d_{\Om}(w)^{\beta}K_{\Om}(w,z)dV(w)\right)K_{\Om}(z,\xi)V(\xi)\\
 		 =&\intop_{\Om}\Big(T_{\alpha, \beta} (K_{\Om}(\cdot,z))(\xi) \Big)K_{\Om}(z,\xi)dV(\xi).
		\end{align}
Thus, by H\"older's inequality and the continuity of $T_{\alpha, \beta}: L_a^p(\Om) \to L_a^q(\Om)$, 
\begin{align}\label{eq-T}
\nonumber &\intop_{\Omega}\left|K_{\Om}\left(w,z\right)\right|^{2}K_{\Om}(w,w)^{-\alpha}d_{\Om}(w)^{\beta}dV(w)\\
\nonumber =&\intop_{\Om}\Big(T_{\alpha, \beta} (K_{\Om}(\cdot,z))(\xi) \Big) d_{\Om}(\xi)^{\frac{a}{q}} K_{\Om}(z,\xi) d_{\Om}(\xi)^{-\frac{a}{q}}dV(\xi)\\
	\nonumber	\leq & \left( \intop_{\Om}\Big|T_{\alpha, \beta} (K_{\Om}(\cdot,z))(\xi) \Big|^q d_{\Om}(\xi)^{a} dV(\xi) \right)^{\frac{1}{q}} \left( \intop_{\Om} |K_{\Om}(z,\xi)|^{q'} d_{\Om}(\xi)^{-\frac{aq'}{q}}dV(\xi)\right)^{\frac{1}{q'}}\\
 		= & \; \no{T_{\alpha, \beta}(K_{\Om}(\cdot,z))}_{q, a}\no{K_{\Om}(\cdot,z)}_{q', -\frac{aq'}{q}} 
 \leq  \; C\no{K_{\Om}(\cdot,z)}_{p, a}\no{K_{\Om}(\cdot,z)}_{q', -\frac{aq'}{q}}
\end{align}
Moreover, using Lemma~\ref{lem-nor-ue} twice for $(p,a)$ with $p > 1, - 1 < a < 2(p - 1)$ and for $(q',-\frac{aq'}{q})$ with $q' > 1, - 1 < -\frac{aq'}{q} < 2(q' - 1)$ by the hypothesis on $p, q, a$, we get that, for every $z \in \Om$,
	\begin{equation*}\label{eqn:Ksquare2}\begin{split}
 	\intop_{\Omega}\left|K_{\Om}\left(w,z\right)\right|^{2}K_{\Om}(w,w)^{-\alpha}d_{\Om}(w)^{\beta}dV(w) 	& \lesssim \; K_{\Om}(z,z)^{1 - \frac{1}{p}}d_{\Om}(z)^{\frac{a}{p}} K_{\Om}(z,z)^{1 - \frac{1}{q'}}d_{\Om}(z)^{-\frac{a}{q}} \\
	& =  \; CK_{\Om}(z,z)^{1-\frac{1}{p}+\frac{1}{q}}d_{\Om}(z)^{-a \left(\frac{1}{q} - \frac{1}{p}\right)}.
 		\end{split}	
 		\end{equation*}

\textbf{Lower estimate.} Similarly to the proof of Lemma \ref{lem-nor-le}, for every $z \in \Om$ near the boundary $\partial \Om$, by the invariant formula, we have 
$$
K_{\Om}(z,w) = \det J_\C\Phi_z(z) K_{\Phi_z(\Om)}(\Phi_z(z), \Phi_z(w)) \overline{\det J_\C\Phi_z(w)},
$$
and
$$
K_{\Om}(w,w) = \det J_\C\Phi_z(w) K_{\Phi_z(\Om)}(\Phi_z(w), \Phi_z(w)) \overline{\det J_\C\Phi_z(w)}.
$$
From this and the fact that $ C^{-1} \leq |\det J_\C\Phi_z(w)|\leq C$ for all $w$ in a neighborhood  of $z$, we get
\begin{align}\label{eq-le-K}
\nonumber & \int_\Om |K_{\Om}(w,z)|^2K_{\Om}(w,w)^{-\alpha}d_{\Om}(w)^{\beta}dV(w) \\
\nonumber = & \int_{\Om'} |\det J_{\C}\Phi_z^{-1}(z')|^{-2}  |K_{\Om'}(w',z')|^2 |\det J_{\C}\Phi_z^{-1}(w')|^{-2} \\
\nonumber & \qquad \qquad \times |\det J_\C\Phi_z^{-1}(w')|^{2\alpha} K_{\Om'}(w',w')^{-\alpha} d_{\Om}(\Phi_z^{-1}(w'))^{\beta} |\det J_{\C}\Phi_z^{-1}(w')| dV(w') \\
 \nonumber \ge & \int_{P_{\lambda}(z')}  |K_{\Om'}(w',z')|^2 K_{\Om'}(w',w')^{-\alpha} d_{\Om}(\Phi_z^{-1}(w'))^{\beta}  |\det J_{\C}\Phi_z(z)|^2 |\det J_{\C}\Phi_z^{-1}(w')|^{2\alpha - 1} dV(w') \\
 \nonumber \gtrsim &\;  \int_{P_{\lambda}(z')} |K_{\Om'}(w',z')|^2K_{\Om'}(w',w')^{-\alpha}d_{\Om'}(w')^{\beta}dV(w')\\
 \nonumber \gtrsim &\;  K_{\Om'}(z',z')^{-\alpha}d_{\Om'}(z')^{\beta}\int_{P_\lambda(z')} |K_{\Om'}(w',z')|^2dV(w')\\
\nonumber \gtrsim &\;  K_{\Om'}(z',z')^{-\alpha}d_{\Om'}(z')^{\beta}K_{\Om'}(z',z')^{2}\T{Vol}(P_\lambda(z'))\\
\gtrsim &\;  K_{\Om'}(z',z')^{1-\alpha}d_{\Om'}(z')^{\beta} \gtrsim K_{\Om}(z,z)^{1-\alpha}d_{\Om}(z)^{\beta}.
\end{align}

Consequently, from the upper and lower estimates of \eqref{eq-sK}, we get
 \begin{equation*} 
  	K_{\Om}(z,z)^{1-\frac{1}{p}+\frac{1}{q}}d_{\Om}(z)^{-a \left(\frac{1}{q} - \frac{1}{p}\right)} \gtrsim K_{\Om}(z,z)^{1-\alpha}d_{\Om}(z)^{\beta},
\end{equation*}
  and hence 
	$$
\sup_{z \in \Om} K_{\Om}(z,z)^{-\alpha + \left(\frac{1}{p} - \frac{1}{q}\right)}d_{\Om}(z)^{\beta + a\left(\frac{1}{q} - \frac{1}{p}\right)} < \infty.
$$
The proof is completed.
\end{proof}

\section{Essential norm}\label{sec-ess-nor}

For a bounded operator $T$ acting from a Banach space $X$ to a Banach space $Y$, the essential norm of $T$ is defined as
$$
\|T\|_{e, X \to Y}: = \inf\{\|T-K\|, K \in \mathcal K(X,Y)\},
$$
where $\mathcal K(X,Y)$ is the set of all compact operators from $X$ to $Y$. Clearly, $T$ is compact if and only if $\|T\|_{e,X \to Y} = 0$. Note that Su\'arez \cite{S07} and Mitkovski, Su\'arez and Wick \cite{MSW13} investigated the essential norm of operators in the Toeplitz algebra on the unit ball in terms of the Berezin transform. Later, \u{C}u\u{c}kovi\'c and \c{S}ahuto\u{g}lu \cite{CS13} established estimates for $\|T_{\psi}\|_{e, L^2(\Om) \to L^2(\Om)}$ with $\psi \in C(\overline{\Om})$ on smooth bounded pseudoconvex domains $\Om$ in $\C^n$, on which the $\overline{\partial}$-Neumann operator is compact.
In this section we give an upper estimate for $\|T_{\psi}\|_{e,  L^p_a(\Om) \to L^q_a(\Om)}$ and a lower one for $\|T_{\alpha, \beta}\|_{e,  L^p_a(\Om) \to L^q_a(\Om)}$.

To do this, we recall that for $1 < p < \infty$, $L^p_a(\Om)$ is a reflexive Banach space and the dual space of $L^p_a(\Om)$ can be identified with the space $L^{p'}_a(\Om)$ under the duality pairing
$$
\langle f, g \rangle_a = \int_{\Om} f(w) \overline{g(w)}dV_a(w) \text{ with } f \in L^p_a(\Om) \text{ and } g \in L^{p'}_a(\Om).
$$

\begin{proposition}\label{prop-com1}
Let $1 < p, q < \infty$, $a > -1$ and $\Om$ be a bounded domain in $\C^n$ such that the Bergman kernel  $K_{\Om}$ is continuous up to the off-diagonal boundary. Suppose that $\psi:\Om\to \C$ is a measurable function in $L^{\infty}(\Om)$ such that $\psi(w) = 0$ almost everywhere on $\Om \setminus Q$ for some compact subset $Q$ of $\Om$. Then the Toeplitz operator $T_{\psi}: L_a^p(\Om) \to L_a^q(\Om)$ is compact.
\end{proposition}
\begin{proof}
Firstly, we show that $T_{\psi}$ maps $L_a^p(\Om)$ to $L_a^q(\Om)$ continuously. Indeed, since 
$$
K_{\Om}\in C\left(\left(\overline{\Omega}\times\overline{\Omega}\right)\setminus \left( \partial\Omega \times \partial\Omega \right) \right),
$$
there exists a positive constant $C$ dependent on $Q$ and $\Om$ such that 
$ \left|K_{\Om}\left(z,w\right)\right|\le C$ for all $(z,w)\in \overline{\Om} \times Q$. 
Then, for every $f \in L_a^p(\Om)$ and $z \in \Om$, using H\"older inequality, we get
\begin{align}\label{eq-Tz}
  \nonumber |T_{\psi}(f)(z)| &= \left| \int_{\Om} K_{\Om}(z,w) \psi(w)f(w)dV(w) \right| \\
 \nonumber & \leq \int_{Q} |K_{\Om}(z,w)| |\psi(w)| |f(w)|dV(w) \leq C \|\psi\|_{\infty} \int_Q |f(w)|dV(w) \\
 \nonumber & \leq C \|\psi\|_{\infty} \left( \int_Q |f(w)|dV_a(w) \right)^{\frac{1}{p}} \left( \int_Q d_{\Om}(w)^{-\frac{ap'}{p}} dV(w) \right)^{\frac{1}{p'}} \\ 
& \leq C \|\psi\|_{\infty} \|f\|_{p,a} \left( \int_Q d_{\Om}(w)^{-\frac{ap'}{p}} dV(w) \right)^{\frac{1}{p'}} < \infty.
\end{align}
Thus, 
$$
\|T_{\psi}f\|_{q,a} \leq C \|\psi\|_{\infty} \|1\|_{q,a} \|f\|_{p,a} \left( \int_Q d_{\Om}(w)^{-\frac{ap'}{p}} dV(w) \right)^{\frac{1}{p'}} \text{ for every } f \in L^p_a(\Om).
$$

Now, we prove that $T_{\psi}: L^p_a(\Om) \to L^q_a(\Om)$ is compact. Since $L^p_a(\Om)$ is reflexive, it is sufficient to prove that for every weakly convergent to $0$ sequence $(f_m)_m$ in $L^p_a(\Om)$, the sequence $T_{\psi}f_m$ converges to $0$ in $L^q_a(\Om)$.
Using \eqref{eq-Tz} for each function $f_m, m \in \mathbb N$, we obtain
\begin{align}\label{eq-bd}
 |T_{\psi}(f_m)(z)|  \leq C \|\psi\|_{\infty} \left( \sup_m \|f_m\|_{p,a} \right) \left( \int_Q d_{\Om}(w)^{-\frac{ap'}{p}} dV(w) \right)^{\frac{1}{p'}} < \infty 
\end{align}
 for every $z \in \Om$, where $\sup_m \|f_m\|_{p,a} < \infty$.
Moreover, for each $z \in \Om$ fixed, it is easy to see that the function 
$$
g_z(w) = K_{\Om}(w,z)\overline{\psi(w)}d_{\Om}(w)^{-a}, w \in \Om,
$$
belongs to $L^{p'}_a(\Om)$. Then, for each $z \in \Om$ fixed,
\begin{align}\label{eq-0}
\nonumber |T_{\psi}(f_m)(z)| &= \left| \int_{\Om} K_{\Om}(z,w) \psi(w)f_m(w)dV(w) \right| \\ 
& = \left| \int_{\Om} f_m(w)\overline{g(w)}dV_a(w) \right| = |\langle f_m, g \rangle_a| \to 0 \text{ as } m \to \infty.
\end{align}
Using \eqref{eq-bd}, \eqref{eq-0} and Lebesgue's dominated convergence theorem, we get
$$
\|T_{\psi}f_m\|_{q,a} = \left(\int_{\Om} |T_{\psi}f_m(z)|^q dV_a(z) \right)^{\frac{1}{q}} \to 0 \text{ as } m \to \infty.
$$
The proof is completed.
\end{proof}

 For each subset $Q$ of $\Om$ and measurable function $\psi: \Om \to \C$, we define 
 $$
\psi_{Q}(w): = \textbf{1}_{Q}(w) \psi(w), w \in \Om,
 $$
 and put
$$
\|M^{\psi}_{Q, p, q, a}\|_{\infty}: =\|M^{\psi_Q}_{\Om, p, q, a}\|_{\infty} = \esssup_{w \in Q} |\psi(w)|K_{\Om}(w,w)^{\frac{1}{p} - \frac{1}{q}}d_{\Om}(z)^{a\left(\frac{1}{q} - \frac{1}{p}\right)}. $$
It is easy to see that for each exhaustion by compact subsets $(Q_m)_m$ of $\Om$ and measurable function $\psi: \Om \to \C$, the limit $\displaystyle \lim_{m \to \infty}\|M^{\psi}_{\Om \setminus Q_m, p, q, a}\|_{\infty}$ exists and does not depend on $(Q_m)_m$. In particular, 
$$
\lim_{m \to \infty}\|M^{\alpha, \beta}_{\Om \setminus Q_m, p, q, a}\|_{\infty} = \limsup_{w \to \partial \Om} K_{\Om}(w,w)^{-\alpha+ \left(\frac{1}{p} - \frac{1}{q}\right)}d_{\Om}(z)^{\beta + a\left(\frac{1}{q} - \frac{1}{p}\right)}.
$$

\begin{theorem}\label{thm:main1:com}
Let $1 < p \leq q < \infty$, $-1 < a < \frac{q}{p'}$, and $\Om$ be a bounded domain in $\C^n$ such that the Bergman kernel $K_{\Om}$ is of sharp $\B$-type. Suppose that $\psi:\Om\to \C$ is a measurable function such that $M^{\psi}_{\Om, p, q, a}(w) \in L^{\infty}(\Om)$. Then 
\begin{equation*}\label{essnorm}
\|T_{\psi}\|_{e, L_a^p(\Om) \to L_a^q(\Om)} \leq C \left(\dfrac{p' + q}{(1+a)\left(1 - \frac{ap'}{q}\right)}\right)^{\frac{1}{p'}+\frac{1}{q}} \lim_{m \to \infty}\|M^{\psi}_{\Om \setminus Q_m, p, q, a}\|_{\infty},
\end{equation*}
for some constant $C$ dependent only on $\Om$. 

In particular, if for every $\varepsilon > 0$ there exists a compact subset $Q = Q(\varepsilon)$ of $\Om$ such that 
$M^{\psi}_{\Om, p, q, a}(w) < \varepsilon$ for almost $w  \in \Om \setminus Q$,
then the Toeplitz operator $T_{\psi}: L^p_{a}(\Om)\to L^q_{a}(\Om)$ is compact.
\end{theorem}
\begin{proof}
By Theorem \ref{thm:main1}, $T_{\psi}: L^p_{a}(\Om)\to L^q_{a}(\Om)$ is bounded.

Let $(Q_m)_m$ be an arbitrary exhaustion by compact sets of $\Om$. By Proposition \ref{prop-com1}, all operators $T_{\psi_{Q_m}}, m \geq 1,$ are compact from $L_a^p(\Om)$ to $L_a^q(\Om)$.

On the other hand,  for every $m \geq 1$ and $f \in L_a^p(\Om)$,
\begin{align*}
T_{\psi}f(z) - T_{\psi_{Q_m}}f(z)&  = \int_{\Om} K_{\Om}(z,w) (\psi(w) - \psi_{Q_m}(w)) f(w)dV(w)\\
& = \int_{\Om} K_{\Om}(z,w) \psi_{\Om\setminus Q_m}(w) f(w)dV(w) = T_{\psi_{\Om\setminus Q_m}}f(z).
\end{align*}

Moreover, it is clear that $M^{\psi_{\Om\setminus Q_m}}_{\Om, p, q, a}(w) \in L^{\infty}(\Om)$. Thus, by Theorem \ref{thm:main1}, for every $m \in \mathbb N$, we get
\begin{align*}
\|T_{\psi}\|_{e, L_a^p(\Om) \to L_a^q(\Om)} & \leq \|T_{\psi} - T_{\psi_{Q_m}}\|_{ L_a^p(\Om) \to L_a^q(\Om)} = \|T_{\psi_{\Om\setminus Q_m}}\|_{ L_a^p(\Om) \to L_a^q(\Om)} \\
& \leq C \left(\dfrac{p' + q}{(1+a)\left(1 - \frac{ap'}{q}\right)}\right)^{\frac{1}{p'}+\frac{1}{q}} \|M^{\psi}_{\Om \setminus Q_m, p, q, a}\|_{\infty}.
\end{align*}
From this the desired estimate follows.
\end{proof}

Next, for each $z \in \Om$, we put 
$$
k_{\Om,a}(w,z) := \dfrac{K_{\Om}(w,z)}{\|K_{\Om}(\cdot,z)\|_{p,a}}, w \in \Om.
$$ 

\begin{lemma}\label{lem-wc-k}
Let $p > 1$, $-1 < a < 2(p - 1)$, and $\Om$ be a bounded smooth pseudoconvex domain in $\C^n$ such that the Bergman kernel is of sharp $\B$-type and $\B$-polydisc condition is satisfied. Then $k_{\Om,a}(\cdot,z)$ converges weakly to $0$ in $L^p_a(\Om)$ as $z \to \partial \Om$. 
\end{lemma}
\begin{proof}
It sufficient to prove that for each $g \in L^{p'}_a(\Om)$, 
$\langle k_{\Om,a}(\cdot, z), g \rangle_a \to 0$  as $z \to \partial \Om$.
Fix an exhaustion by compact subsets $(Q_m)_m$ of $\Om$. For each $m \in \mathbb N$ fixed, since 
$
K_{\Om}\in C\left(\left(\overline{\Omega}\times\overline{\Omega}\right)\setminus \left( \partial \Om \times \partial\Omega\right) \right),
$ 
there exists a positive constant $C_m$ such that 
$\left|K_{\Om}\left(w, z\right)\right|\le C_m$ for all $(w,z)\in Q_m \times \overline{\Om}$. 

Using H\"older's inequality and Lemma \ref{lem-nor-le}, for every $z$ near the boundary $\partial \Om$, we get
\begin{align*}
\left| \langle k_{\Om,a}(\cdot, z), g \rangle_a \right| & = \left| \int_{\Om} k_{\Om, a}(w,z) \overline{g(w)}dV_a(w) \right|  \leq  \int_{\Om} |k_{\Om, a}(w,z)| |g(w)| dV_a(w) \\
& = \dfrac{1}{\|K_{\Om}(\cdot,z)\|_{p,a}} \int_{Q_m} |K_{\Om}(w,z)| |g(w)| dV_a(w) + \int_{\Om \setminus Q_m} |k_{\Om,a}(w,z)| |g(w)| dV_a(w)  \\
& \leq \dfrac{C_m}{\|K_{\Om}(\cdot,z)\|_{p,a}} \int_{Q_m} |g(w)| dV_a(w) \\
& + \left( \int_{\Om \setminus Q_m} |k_{\Om, a}(w,z)|^p dV_a(w) \right)^{\frac{1}{p}}  \left( \int_{\Om \setminus Q_m} |g(w)|^{p'} dV_a(w) \right)^{\frac{1}{p'}} \\
& \lesssim \; C_m  K_{\Om}(z,z)^{\frac{1}{p} - 1} d_{\Om}(z)^{-\frac{a}{p}}  \int_{Q_m} |g(w)| dV_a(w) +  \left( \int_{\Om \setminus Q_m} |g(w)|^{p'} dV_a(w) \right)^{\frac{1}{p'}} \\
& \lesssim C_m d_{\Om}(z)^{2 - \frac{a + 2}{p}}  \int_{Q_m} |g(w)| dV_a(w) +  \left( \int_{\Om \setminus Q_m} |g(w)|^{p'} dV_a(w) \right)^{\frac{1}{p'}},
\end{align*}
since, by \eqref{eq-ine-dk}, $d_{\Om}^{-2}(z) \leq K_{\Om}(z,z)$ for every $z \in \Om$.

In the last inequality letting first $z \to \partial \Om$, and then $m \to \infty$, we can conclude that $\langle k_{\Om}(\cdot, z), g \rangle_a \to 0$ as $z \to \partial \Om$.
\end{proof}

\begin{theorem}\label{thm:main2:com}
Let $1 < p \leq q < \infty$, $-1 < a < \min\{2(p - 1), q-1\}$, and $\Om$ be a bounded smooth pseudoconvex domain in $\mathbb C^n$ such that the Bergman kernel is of sharp $\mathcal B$-type and $\B$-polydisc condition is satisfied. Suppose that $T_{\alpha, \beta}: L_a^p(\Om) \to L_a^q(\Om)$ with $\alpha, \beta \geq 0$ is continuous.
Then
$$
\|T_{\alpha, \beta} \|_{e, L^p_a(\Om) \to L^q_a(\Om)} \gtrsim \limsup_{w \to \partial \Om} M^{\alpha, \beta}_{\Om, p, q, a}(w).
$$
\end{theorem}
\begin{proof} 
Let $T$ be an arbitrary compact operator from $L^p_a(\Om)$ to $L^q_a(\Om)$. Then, by Lemma \ref{lem-wc-k}, $Tk_{\Om,a}(\cdot, z) \to 0$ in $L^q_a(\Om)$ as $z \to \partial \Om$.

Using \eqref{eq-T}, \eqref{eq-le-K}, and Lemma \ref{lem-nor-ue}, for every $z \in \Om$ near the boundary $\partial \Om$, we get 
\begin{align}\label{eq-lo-tk}
\nonumber \|T_{\alpha, \beta} k_{\Om, a}(\cdot, z)\|_{q, a} & = \dfrac{\|T_{\alpha, \beta} K_{\Om}(\cdot, z)\|_{q, a}}{\|K_{\Om}(\cdot, z)\|_{p, a}} \\
\nonumber &\geq \dfrac{1}{\|K_{\Om}(\cdot, z)\|_{p, a} \|K_{\Om}(\cdot, z)\|_{q', -\frac{aq'}{q}}} \int_{\Om} |K_{\Om}(w,z)|^2 K_{\Om}(w,w)^{-\alpha}d_{\Om}(w)^{\beta} dV(w) \\
\nonumber & \gtrsim K_{\Om}(z,z)^{\frac{1}{p}-1} d_{\Om}(z)^{-\frac{a}{p}} K_{\Om}(z,z)^{\frac{1}{q'}-1}d_{\Om}(z)^{\frac{a}{q}} K_{\Om}(z,z)^{1 - \alpha}d_{\Om}(z)^{\beta} \\
& = K_{\Om}(z,z)^{-\alpha + \left(\frac{1}{p}-\frac{1}{q}\right)} d_{\Om}(z)^{\beta + a \left(\frac{1}{q} - \frac{1}{p}\right)} = M^{\alpha, \beta}_{\Om, p, q, a}(z).
\end{align}

Consequently, for every $z$ near the boundary $\partial \Om$,
\begin{align*}
\|T_{\alpha, \beta} - T\|_{L_a^p(\Om) \to L_a^q(\Om)} & \geq \|T_{\alpha, \beta}k_{\Om, a}(\cdot, z) - Tk_{\Om, a}(\cdot, z)\|_{q,a} \\
& \geq \|T_{\alpha, \beta}k_{\Om, a}(\cdot, z)\|_{q,a} - \|Tk_{\Om, a}(\cdot, z)\|_{q,a} \\
& \gtrsim M^{\alpha, \beta}_{\Om, p, q, a}(z) - \|Tk_{\Om}(\cdot, z)\|_{q,a}.
\end{align*}
Letting $z \to \partial \Om$ in the last inequality, we get
$$
\|T_{\alpha, \beta} - T\|_{L_a^p(\Om) \to L_a^q(\Om)} \gtrsim \limsup_{z \to \infty} M^{\alpha, \beta}_{\Om, p, q, a}(z),
$$
which implies the desired estimate.
\end{proof}

From Theorems \ref{thm:main1:com} and \ref{thm:main2:com}, we immediately get the following result for the case $p = q$ and $a = 0$.

\begin{corollary}\label{cor:com}
Let $1 < p < \infty$ and $\Om$ be a bounded smooth pseudoconvex domain in $\mathbb C^n$ such that the Bergman kernel is of sharp $\mathcal B$-type and $\B$-polydisc condition is satisfied. The operator $T_{\alpha, \beta}$ with $\alpha, \beta \geq 0$ is compact on $L^p(\Om)$ if and only if $\alpha + \beta > 0$.
\end{corollary}

\section{Schatten class Toeplitz operators}\label{sec-sch}
In this section we establish a characterization of Schatten class membership of Toeplitz operators on $L^2(\Om)$. It should be noted that this characterization has been investigated only on the unit ball (see, \cite{APP15, P14, Z07}) and has not been considered before on pseudoconvex domains whose Bergman kernel has no an explicit formula.

Recall that for $0 < s < \infty$, a compact operator $T$ acting on a separable Hilbert space $H$ belongs to the Schatten class $S_s$ if its sequence of singular numbers belongs to the sequence space $\ell^s$, where the singular numbers are the square roots of the eigenvalues of the positive operator $T^* T$ , where $T^*$ is the Hilbert adjoint of $T$.

For simplicity, we write $k_{\Om}$ and $\langle \cdot, \cdot \rangle$ instead of $k_{\Om, 0}$ and $\langle \cdot, \cdot \rangle_0$, respectively.
For a positive operator on $L^2(\Om)$, the Berezin transform of the operator $T$ is defined by 
$$
\widetilde{T}(z): = \langle T k_{\Om}(\cdot, z), k_{\Om}(\cdot, z) \rangle, z \in \Om.
$$
The following auxiliary lemmas are elementary, hereby we sketch the proofs for the sake of the completeness.

\begin{lemma}\label{lem-BT}
Let $\Om$ be a bounded domain in $\C^n$ and $T$ a positive compact operator on $L^2(\Om)$. The following statements are valid:
\begin{itemize}
\item[(a)] For $s \geq 1$, if $T$ is in $S_s$, then $K_{\Om}(z,z) \widetilde{T}(z)^s  \in L^1(\Om)$.
\item[(b)] For $0 < s \leq 1$, if $K_{\Om}(z,z) \widetilde{T}(z)^s  \in L^1(\Om)$, then $T$ is in $S_s$.
\end{itemize}
\end{lemma}
\begin{proof}
 By \cite[Lemma~1.25]{Zhu07}, the positive compact operator $T \in S_s$  if and only if $T^s \in S_1$. 

Let $(e_m)_m$ be an arbitrary orthonormal basis in $L^2(\Om)$. Since
$$
\sum_{m=1}^{\infty} \overline{e_m(z)} e_m  = \sum_{m=1}^{\infty} \langle K_{\Om}(\cdot,z), e_m \rangle e_m  = K_{\Om}(\cdot,z) \text{ for every } z \in \Om,
$$ 
we get
\begin{align*}
\text{tr}(T^s) & = \sum_{m=1}^{\infty} \langle T^se_m, e_m \rangle = \sum_{m=1}^{\infty} \int_{\Om} T^s e_m(z) \overline{e_m(z)}dV(z)\\
& = \sum_{m=1}^{\infty} \int_{\Om} \langle T^s e_m, K_{\Om}(\cdot, z) \rangle \overline{e_m(z)}dV(z) = \sum_{m=1}^{\infty} \int_{\Om} \left \langle T^s  \overline{e_m(z)} e_m, K_{\Om}(\cdot, z) \right \rangle dV(z) \\
& = \int_{\Om} \left \langle T^s \left(\sum_{m=1}^{\infty} \overline{e_m(z)} e_m \right), K_{\Om}(\cdot, z) \right \rangle dV(z) = \int_{\Om} \left \langle T^s K_{\Om}(\cdot, z), K_{\Om}(\cdot, z) \right \rangle dV(z) \\
& = \int_{\Om} \left \langle T^s k_{\Om}(\cdot, z), k_{\Om}(\cdot, z) \right \rangle  \|K_{\Om}(\cdot,z)\|^2_2 dV(z) = \int_{\Om} \left \langle T^s k_{\Om}(\cdot, z), k_{\Om}(\cdot, z) \right \rangle K_{\Om}(z,z)dV(z).
\end{align*} 

Moreover, by \cite[Proposition~1.31]{Zhu07},  for every $z \in \Om$, we have
$$
\left \langle T^s k_{\Om}(\cdot, z), k_{\Om}(\cdot, z) \right \rangle \geq \left \langle T k_{\Om}(\cdot, z), k_{\Om}(\cdot, z) \right \rangle^s = \widetilde{T}(z)^s \text{ if } s \geq 1,
$$
and
$$
\left \langle T^s k_{\Om}(\cdot, z), k_{\Om}(\cdot, z) \right \rangle \leq \left \langle T k_{\Om}(\cdot, z), k_{\Om}(\cdot, z) \right \rangle^s = \widetilde{T}(z)^s \text{ if } 0 < s \leq 1.
$$

Consequently, both assertions (a) and (b) follow from the above inequalities.
\end{proof}

\begin{lemma}\label{lem-sc-sch}
Let $s \geq 1$ and $\Om$ be a bounded domain in $\C^n$. Suppose that $\psi$ is a positive function in $L^{\infty}(\Om)$ such that $T_{\psi}$ is compact on $L^2(\Om)$. If $K_{\Om}(w,w)\psi(w)^s \in L^1(\Om)$, then $T_{\psi}$ belongs to $S_s$.
\end{lemma}
\begin{proof}
Let $(e_m)_m$ be an arbitrary orthonormal set in $L^2(\Om)$. Then for every $m \in \mathbb N$, using Fubini's theorem, we obtain
\begin{align*}
\langle T_{\psi} e_m, e_m \rangle & = \int_{\Om} \left( \int_{\Om} K_{\Om}(z,w)\psi(w)e_m(w)dV(w)\right) \overline{e_m(z)}dV(z) \\
& = \int_{\Om} \left( \overline{\int_{\Om} K_{\Om}(w, z)e_m(z)dV(z)}\right) \psi(w) e_m(w)dV(w) = \int_{\Om} |e_m(w)|^2 \psi(w)dV(w).
\end{align*}
Thus, using H\"older's inequality and the inequality (based on the Bessel's inequality) 
$$
\sum_{m = 1}^{\infty}|e_m(w)|^{2} = \sum_{m = 1}^{\infty}|\langle K_{\Om}(\cdot,w), e_m \rangle|^{2} \leq \|K_{\Om}(\cdot, w)\|_2^2,
$$
we get
\begin{align*}
\sum_{m = 1}^{\infty}\langle T_{\psi} e_m, e_m \rangle^s & \leq  \sum_{m = 1}^{\infty} \int_{\Om} |e_m(w)|^{2} \psi(w)^s dV(w) \\
& \leq  \int_{\Om} \left(\sum_{m = 1}^{\infty}|e_m(w)|^{2}\right) \psi(w)^s dV(w) \\
& \leq  \int_{\Om} \|K_{\Om}(\cdot,w)\|_2^{2} \psi(w)^s dV(w) = \int_{\Om} K_{\Om}(w,w) \psi(w)^s  dV(w).
\end{align*}
From this and \cite[Theorem~1.27]{Zhu07}, the assertion follows.
\end{proof}

\begin{theorem}\label{thm:main:sch}
Let $\Om$ be a bounded smooth pseudoconvex domain in $\mathbb C^n$ such that the Bergman kernel is of sharp $\mathcal B$-type and $\B$-polydisc condition  is satisfied. Suppose that the operator $T_{\alpha, \beta}$ with $\alpha, \beta \geq 0$ is compact on $L^2(\Om)$. The following statements hold:
\begin{itemize}
\item[(a)] For $s \geq 1$, $T_{\alpha, \beta} \in S_s$ if and only if $K_{\Om}(w,w)^{1 - s\alpha} d_{\Om}(w)^{s\beta} \in L^1(\Om)$.
\item[(b)] For $s \in (0, 1)$, if $2\alpha + \beta < 2$ and $K_{\Om}(w,w)^{1 - s\alpha} d_{\Om}(w)^{s\beta} \in L^1(\Om)$, then $T_{\alpha, \beta} \in S_s$.
\end{itemize}
\end{theorem}
\begin{proof}
(a) The sufficiency follows immediately from Lemma \ref{lem-sc-sch}. 

\textbf{Necessity.} Obviously, $T_{\alpha, \beta}$ is positive operator on $L^2(\Om)$. 
Moreover, for every $z \in \Om$ near the boundary $\partial \Om$, by \eqref{eqn:Ksquare1} and \eqref{eq-le-K}, we get
\begin{align*}
\nonumber \widetilde{T_{\alpha, \beta}}(z) & = \left \langle T_{\alpha, \beta} k_{\Om}(\cdot, z), k_{\Om}(\cdot, z) \right \rangle = \|K_{\Om}(\cdot,z)\|_2^{-2} \left \langle T_{\alpha, \beta} K_{\Om}(\cdot, z), K_{\Om}(\cdot, z) \right \rangle \\
\nonumber & =  K_{\Om}(z,z)^{-1} \int_{\Om} (T_{\alpha,\beta} K_{\Om}(\cdot, z))(\xi) K_{\Om}(z, \xi)dV(\xi) \\
 & = K_{\Om}(z,z)^{-1} \int_{\Om} |K_{\Om}(w,z)|^2 K_{\Om}(w,w)^{-\alpha} d_{\Om}(w)^{\beta}dV(w) \\
\nonumber & \gtrsim K_{\Om}(z,z)^{-\alpha}d_{\Om}(z)^{\beta}.
\end{align*}
From this and Lemma \ref{lem-BT}(a), the assertion follows.

(b) Since $0 \leq 2\alpha + \beta < 2$, by Lemma A in Appendix, 
$$
\int_{\Om} |K_{\Om}(w,z)|^2 K_{\Om}(w,w)^{-\alpha} d_{\Om}(w)^{\beta}dV(w) \lesssim K_{\Om}(z,z)^{1-\alpha} d_{\Om}(z)^{\beta}, 
$$
for every $z \in \Om$. Thus, for each $z \in \Om$, we obtain
\begin{align*}
\widetilde{T_{\alpha, \beta}}(z) & =  K_{\Om}(z,z)^{-1} \int_{\Om} |K_{\Om}(w,z)|^2 K_{\Om}(w,w)^{-\alpha} d_{\Om}(w)^{\beta}dV(w) \lesssim  K_{\Om}(z,z)^{-\alpha} d_{\Om}(z)^{\beta}.
\end{align*}
From this and Lemma \ref{lem-BT}(b), the assertion follows.
\end{proof}

\section{Toeplitz operators between weighted Bergman spaces}\label{sec-Ber}

In this section we consider the Toeplitz operator $T_{\psi}$ acting from a weighted Bergman space $A_a^p(\Om)$ to another one $A_a^q(\Om)$ with $1 < p \leq q < \infty$. 
We recall that the weighted Bergman space $A^p_{a}(\Om) := L^p_{a}(\Om) \cap H(\Om)$, where $H(\Om)$ is the space of all holomorphic functions on $\Om$ endowed with the usual compact open topology $co$. Using the plurisubharmonicity of $|f(z)|^p$ with $f \in A_a^p(\Om)$, we can see that the topology induced by $\|\cdot\|_{p,a}$ is stronger than $co$ in $A_a^p(\Om)$.

\subsection{Boundedness and compactness} Since $T_{\psi}$ acts from $L^p_a(\Om)$ into $H(\Om)$, Theorem \ref{thm:main1} and Proposition \ref{prop-com1} also hold for the operator $T_{\psi}: A_a^p(\Om) \to A_a^q(\Om)$, and hence, so does Theorem \ref{thm:main1:com}, i.e. the upper estimate for $\|T_{\psi}\|_{e, L^p_a(\Om) \to L^q_a(\Om)}$ obtained in Theorem \ref{thm:main1:com} is valid for $\|T_{\psi}\|_{e, A^p_a(\Om) \to A^q_a(\Om)}$.

Moreover, note that in the proof of Theorem \ref{thm:main2} we used the inequality 
$$
\no{T_{\alpha, \beta}(K_{\Om}(\cdot,z))}_{q, a} \leq C \no{K_{\Om}(\cdot,z)}_{p, a}
$$
for holomorphic functions $K_{\Om}(\cdot,z)$. Thus, in Theorem \ref{thm:main2} we can replace weighted $L^p$-spaces $L^p_{a}(\Om)$ and $L^q_{a}(\Om)$ by the corresponding weighted Bergman spaces $A^p_{a}(\Om)$ and $A^q_{a}(\Om)$, respectively, to get necessary conditions for boundedness of $T_{\alpha, \beta}: A^p_{a}(\Om) \to A^q_{a}(\Om)$. 

On the other hand, since, in general, Lemma \ref{lem-wc-k} may be false for the space $A_a^p(\Om)$, we cannot get the lower estimate for $\|T_{\alpha,\beta}\|_{e, A_a^p(\Om) \to A_a^q(\Om)}$ as in Theorem \ref{thm:main2:com}. However, we can obtain the following necessary condition for compactness of $T_{\alpha, \beta}: A^p_{a}(\Om) \to A^q_{a}(\Om)$. 

\begin{theorem}
Let $1 < p \leq q < \infty$, $-1 < a < \min\{2(p - 1), q-1\}$, and $\Om$ be a bounded smooth pseudoconvex domain in $\C^n$ such that the Bergman kernel is of sharp $\B$-type and $\B$-polydisc condition is satisfied. 
If the operator $T_{\alpha, \beta}: A^p_{a}(\Om) \to A^q_{a}(\Om)$ with $\alpha, \beta \geq 0$ is compact, then 
$$
\limsup_{w \to \partial \Om} M_{\Om, p, q, a}^{\alpha, \beta}(w) = 0.
$$
\end{theorem}
\begin{proof}
By \eqref{eq-lo-tk}, $\|T_{\alpha, \beta}k_{\Om,a}(\cdot,z)\|_{q,a} \gtrsim  M_{\Om, p, q, a}^{\alpha, \beta}(z) $ for every $z \in \Om$ near the boundary $\partial \Om$. Thus, it is enough to prove that $T_{\alpha, \beta}k_{\Om,a}(\cdot, z) \to 0$ in $A_a^q(\Om)$ as $z \to \partial \Om$. 

First, we show that
$k_{\Om,a}(\cdot,z)$ converges to $0$ in $H(\Om)$ as $z \to \partial \Om$. Indeed, for each compact subset $Q$ of $\Om$, using the continuity up to the off-diagonal boundary of the Bergman kernel and Lemma \ref{lem-nor-le}, we get
\begin{align*}
\sup_{w \in Q} |k_{\Om,a}(w,z)| & \lesssim K_{\Om}(z,z)^{\frac{1}{p} - 1} d_{\Om}(z)^{-\frac{a}{p}} \sup_{w \in Q} |K_{\Om}(w,z)| \\
& \lesssim d_{\Om}(z)^{1 - \frac{a + 1}{p}} \sup_{w \in Q, z \in \overline{\Om}} |K_{\Om}(w,z)|  \to 0 \text{ as } z \to \partial \Om,
\end{align*}
 since, by \eqref{eq-ine-dk}, $d_{\Om}^{-2}(z) \leq K_{\Om}(z,z)$ and $a < 2(p - 1)$.

Next, by contradiction, we assume that there is a sequence $(z_m)_m$ in $\Om$ such that $z_m \to \partial \Om$ and $\|T_{\alpha, \beta}k_{\Om,a}(\cdot,z_m)\|_{q,a} \geq \delta$ for every $m \in \mathbb N$ and some number $\delta > 0$.
However, since $\|k_{\Om,a}(\cdot,z_m)\|_{p,a} = 1, m \in \mathbb N$, and $T_{\alpha, \beta}: A^p_{a}(\Om) \to A^q_{a}(\Om)$ is compact, we can suppose that the sequence $T_{\alpha, \beta}k_{\Om,a}(\cdot,z_m)$ converges to some function $g$ in $A_a^q(\Om)$, and hence, in $H(\Om)$. 
Now we claim that $T_{\alpha, \beta}k_{\Om,a}(\cdot,z_m) \to 0$ in $H(\Om)$. Then $g$ must be the zero function which is a contradiction.

To prove the claim, we fix an exhaustion by compact subsets $(Q_j)_j$ of $\Om$ and an arbitrary compact subset $Q$ of $\Om$. For simplicity, we put $h_m(w): = k_{\Om,a}(w,z_m)$, $m \in \mathbb N$. Then $\|h_m\|_{p, a} = 1$ and $h_m \to 0$ in $H(\Om)$ as $m \to \infty$. 
Since 
$
K_{\Om}\in C\left(\left(\overline{\Omega}\times\overline{\Omega}\right)\setminus \left( \partial \Om \times \partial\Omega\right) \right),
$ 
there exists a positive constant $C$ such that 
$\left|K_{\Om}\left(z,w\right)\right|\le C$ for all $(z,w)\in Q \times \overline{\Om}$.
Moreover, by Theorem \ref{thm:main2}, 
\begin{equation*}
K_{\Om}(w,w)^{-\alpha}d_{\Om}(w)^{\beta} \lesssim K_{\Om}(w,w)^{\frac{1}{q} - \frac{1}{p}}d_{\Om}(w)^{a\left(\frac{1}{p} - \frac{1}{q}\right)} \text{ for every } w \in \Om.
\end{equation*}
Using these inequalities, H\"older's inequality, and \eqref{eq-ine-dk}, we get that for every $m, j \in \mathbb N$,
\begin{align*}
& \sup_{z \in Q} \left|T_{\alpha, \beta}h_m(z) \right| 
=  \sup_{z \in Q} \left|\int_{\Om} K_{\Om}(z,w) K_{\Om}(w,w)^{-\alpha}d_{\Om}(w)^{\beta}h_{m}(w)dV(w)\right| \\
\lesssim &\; C \int_{\Om} |h_{m}(w)| K_{\Om}(w,w)^{\frac{1}{q}-\frac{1}{p}}d_{\Om}(w)^{a\left(\frac{1}{p} - \frac{1}{q}\right)} dV(w) \\
= &\; C \left( \int_{Q_j} +  \int_{\Om\setminus Q_j}\right) |h_{m}(w)| K_{\Om}(w,w)^{\frac{1}{q}-\frac{1}{p}}d_{\Om}(w)^{a\left(\frac{1}{p} - \frac{1}{q}\right)} dV(w) \\
\leq &\; C \sup_{w \in Q_j} |h_{m}(w)| \sup_{w \in Q_j} K_{\Om}(w,w)^{\frac{1}{q}-\frac{1}{p}}d_{\Om}(w)^{a\left(\frac{1}{p} - \frac{1}{q}\right)}  \\
+ & \; C \left( \int_{\Om\setminus Q_j}|h_{m}(w)|^p d_{\Om}(w)^{a}dV(w) \right)^{\frac{1}{p}} \times  \left( \int_{\Om\setminus Q_j} K_{\Om}(w,w)^{p'\left(\frac{1}{q}-\frac{1}{p}\right)}d_{\Om}(w)^{-\frac{ap'}{q}}dV(w)\right)^{\frac{1}{p'}} \\
\leq &\; C \sup_{w \in Q_j} |h_{m}(w)| \sup_{w \in Q_j} K_{\Om}(w,w)^{\frac{1}{q}-\frac{1}{p}}d_{\Om}(w)^{a\left(\frac{1}{p} - \frac{1}{q}\right)}  
+  \; C  \left( \int_{\Om\setminus Q_j} d_{\Om}(w)^{p'\left(\frac{2}{p}-\frac{a+2}{q} \right)}dV(w)\right)^{\frac{1}{p'}}.
\end{align*}
In the last inequality letting first $m \to \infty$, and then $j \to \infty$, we obtain
$$
\lim_{m \to \infty}\sup_{z \in Q} \left|T_{\alpha, \beta}h_m(z) \right| \lesssim \limsup_{j \to \infty} \left( \int_{\Om\setminus Q_j} d_{\Om}(w)^{p'\left(\frac{2}{p}-\frac{a+2}{q} \right)}dV(w)\right)^{\frac{1}{p'}} = 0,
$$
since $p'\left(\frac{2}{p}-\frac{a+2}{q}\right) > -1$ by hypothesis on $p, q, a$, and hence, $d_{\Om}(w)^{p'\left(\frac{2}{p}-\frac{a+2}{q} \right)} \in L^1(\Om)$.
\end{proof}

\begin{remark}
If, in addition, the continuous dual of $A^p_a(\Om)$ is $A_a^{p'}(\Om)$, then Lemma \ref{lem-wc-k} holds for this space $A^p_a(\Om)$, and hence, so does the lower estimate in Theorem \ref{thm:main2:com} for $\|T_{\alpha, \beta}\|_{e, A^p_a(\Om) \to A^q_a(\Om)}$.
\end{remark}

\subsection{Schatten class Toeplitz operators}
Since the arguments in Section \ref{sec-sch} are based on the Berezin transform $\widetilde{T}(z) = \langle Tk_{\Om}(\cdot,z), k_{\Om}(\cdot,z)\rangle$ with $k_{\Om}(\cdot,z) \in A^2(\Om)$. Then we can repeat these arguments for the operator $T_{\psi}$ on $A^2(\Om)$ to show that Lemmas \ref{lem-BT} and \ref{lem-sc-sch}, and Theorem \ref{thm:main:sch} also hold for $T_{\psi}$ on $A^2(\Om)$.

To end this section, we summarize all results for the Toeplitz operator $T_{\alpha, \beta}: A^p_{a}(\Om) \to A^q_{a}(\Om)$.
\begin{theorem}\label{thm:main:B}
Let $1 < p \leq q < \infty$, $-1 < a < \min\{2(p-1), \frac{q}{p'} \}$, and $\Om$ be a bounded smooth pseudoconvex domain in $\C^n$ such that the Bergman kernel is of sharp $\B$-type and $\B$-polydisc condition is satisfied. For every $\alpha, \beta \geq 0$, the following statements hold:
\begin{itemize}
\item[(1)] The operator $T_{\alpha, \beta}: A_a^p(\Om) \to A_a^q(\Om)$ is continuous if and only if $M^{\alpha, \beta}_{\Om, p, q, a}(w) \in L^{\infty}(\Om)$.
In this case, 
$$
\|T_{\alpha, \beta}\|_{A_a^p(\Om) \to A_a^q(\Om)} \leq C \left( \dfrac{p'+q}{(1+a)\left(1 - \frac{ap'}{q} \right)} \right)^{\frac{1}{p'} + \frac{1}{q}} \|M^{\alpha, \beta}_{\Om, p, q, a}\|_{\infty},
$$
where $C$ is independent of $p, q, a$.
\item[(2)] The operator $T_{\alpha, \beta}: A_a^p(\Om) \to A_a^q(\Om)$ is compact if and only if 
\begin{equation*}
\limsup_{w \to \partial \Om} M^{\alpha, \beta}_{\Om, p, q, a}(w) = 0.
\end{equation*}
\item[(3)] Suppose that the operator $T_{\alpha, \beta}$ is compact on $A^2(\Om)$. For every $s \geq 1$, the operator $T_{\alpha, \beta}$ belongs to Schatten class $S_s$ if and only if $K_{\Om}(w,w)^{1 - s\alpha}d_{\Om}(w)^{s\beta} \in L^1(\Om)$. In the case $s \in (0,1)$, if $2\alpha + \beta < 2$ and $K_{\Om}(w,w)^{1 - s\alpha}d_{\Om}(w)^{s\beta} \in L^1(\Om)$, then $T_{\alpha, \beta}$ is in $S_s$.
\end{itemize}
\end{theorem}

\section*{Appendix}
In this section we prove a generalization of \cite[Proposition~2.4]{KLT18} and another upper estimate for quantity \eqref{eq-sK} without using the continuity of the operator $T_{\alpha, \beta}: L_a^p(\Om) \to L_a^q(\Om)$. These results may have some independent interest.

\medskip

\textbf{Proposition A.} \textit{Let $\Om$ be a domain in $\C^n$ such that the Bergman kernel $K_{\Om}$ is of sharp $\mathcal{B}$-type. Then, for each $z_0\in \partial\Om$, there is a neighbourhood $U$ of $z_0$ such that  for any $s \geq 0, a - s \ge 1$ and $-1 < b + 2s <2a-2$, 
	\begin{equation*}
	\begin{split}
	I_{a,b,s}\left(z\right) &:= \intop_{\Omega\cap U}\left|K_{\Om}\left(z,w\right)\right|^{a} d_{\Om}\left(w\right)^{b}K_{\Om}(w,w)^{-s}dV(w) \\
		&\leq C\frac{2a-1}{(2a-2-b - 2s)(b+ 2s +1)}K_{\Om}(z,z)^{a-s-1}d_{\Om}\left(z\right)^{b}
	\end{split}
	\end{equation*}
	for every $z\in\Omega\cap U$ and some constant $C$ dependent only on $U$ and $\Om$.}

\begin{proof} Similarly to \cite[Proposition~2.4]{KLT18}, we choose  a small neighbourhood $U$ of $z_0$ such that $\Phi_z(U)\subset \BB(z',c)$ for any $z\in U$,  where  the ball $\BB(z',c)$ and the biholomorphism $\Phi_z$ are given in Definition~\ref{df-bsharp}. Using the invariant formula as in the proof of Theorem \ref{thm:main2}, we get
$$
I_{a,b,s}(z)\le C\int_{\Om'\cap \BB(z',c)}|K_{\Om'}(z',w')|^{a}d_{\Om'}(w')^{b}K_{\Om'}(w',w')^{-s}dV(w').
$$
Thus, it suffices to show that 
	\begin{equation}\label{eqn:Iab prime}
	\begin{split}
I'_{a,b,s}\left(z'\right) &:=\intop_{\Omega'\cap \BB(z',c)}\left|K_{\Om'}\left(z',w'\right)\right|^{a}d_{\Om'}\left(w'\right)^{b}K_{\Om'}(w',w')^{-s}dV(w') \\
		& \leq C\frac{2a-1}{(2a-2-b-2s)(b+2s+1)}K_{\Om'}(z',z')^{a-s-1}d_{\Om'}\left(z'\right)^{b},
	\end{split}
	\end{equation} 
for every $z' \in \Om' \cap U'$. 

It is clear that $b_j(z',w')\le b_j(z',z')$  and $b_1(z',w')\le \dfrac{1}{d_{\Om}(z')+d_{\Om}(w')}$. Since $K_{\Om'}$ is of sharp $\B$-type at $z'$, for all $w'\in \Om'\cap \BB(z',c)$, we have
	\begin{eqnarray} 	\label{eqn:Kzw Kzz}
	|K_{\Om'}(z',w')|\le C K_{\Om'}(z',z'){\left(\frac{d_{\Om'}(z')}{d_{\Om'}(z')+d_{\Om'}(w')}\right)^{2}},
	\end{eqnarray}
	and
	\begin{equation}\label{eqn:Kww Kzw}
           K_{\Om}(w',w')^{-1} \leq C |K_{\Om}(z',w')|^{-1}{\left(\frac{d_{\Om'}(w')}{d_{\Om'}(w')+d_{\Om'}(z')}\right)^{2}}.
	\end{equation}
Since $s \geq 0$ and $a - s\geq1$, from \eqref{eqn:Kzw Kzz} and \eqref{eqn:Kww Kzw} it follows that
\begin{align*}
I'_{a,b,s}(z') & \leq \intop_{\Omega'\cap \BB(z',c)}\left|K_{\Om'}\left(z',w'\right)\right|^{a-s} \dfrac{d_{\Om'}\left(w'\right)^{b+2s}}{(d_{\Om'}(w')+d_{\Om'}(z'))^{2s}}dV(w') \\
& \leq K_{\Om'}(z',z')^{a-s-1}d_{\Om'}(z')^{2(a-s-1)} \intop_{\Omega'\cap \BB(z',c')}\left|K_{\Om'}\left(z',w'\right)\right| \dfrac{d_{\Om'}\left(w'\right)^{b+2s}}{(d_{\Om'}(w')+d_{\Om'}(z'))^{2a - 2}}dV(w').
\end{align*} 
Moreover, using the estimates in the proof of \cite[Proposition 2.4]{KLT18} with $- 1 < b + 2s < 2a - 2$, we obtain
	\begin{align*}
		 J_{a,b,s}(z')&=\int_{\Om'}|K_{\Om'}(z',w')|\frac{d_{\Om'}(w')^{b+2s}}{(d_{\Om'}(z')+d_{\Om'}(w'))^{2a-2}}dV(w') \\
		& \leq C \dfrac{2a-1}{(2a-2-b-2s)(b+2s+1)}d_{\Om'}(z')^{b+2s - 2a +2}.
	\end{align*}

Consequently, from these estimates \eqref{eqn:Iab prime} follows.
\end{proof}

Similarly to Lemma \ref{lem-nor-ue}, from Proposition A we can get the following upper estimate for quantity \eqref{eq-sK}.

\medskip

\textbf{Lemma A.} \textit{Let $\Om$ be a domain in $\C^n$ such that  the Bergman kernel $K_{\Om}$ is of sharp $\mathcal{B}$-type. For every $\alpha, \beta \geq 0$ with $2\alpha + \beta < 2$, the following inequality 
	\begin{equation*}
\int_{\Om} |K_{\Om}(z,w)|^2 K_{\Om}(w,w)^{-\alpha}d_{\Om}(w)^{\beta}dV(w) \lesssim K_{\Om}(z,z)^{1 - \alpha}d_{\Om}(z)^{\beta}
	\end{equation*}
	holds for every $z\in\Omega$.}

\begin{proof}
As in the proof of Lemma \ref{lem-nor-ue}, we choose a covering $\{U_j\}_{j=0}^N$ to $\overline\Om$ so that $U_{0}\Subset \Omega$,  $\partial\Omega\subset\bigcup_{j=1}^{N}U_{j}$, and the integral estimates in Proposition A hold on $U_j$ with some constant $C_j$ for all $j=1,\dots, N$.

Since $K_\Om\in C((\overline\Om\times\overline\Om) \setminus (\partial \Om \times \partial \Om) )$, there is a constant $C > 0$ such that
$$
\left|K_{\Om}\left(w,z\right)\right|\le C
\quad \T{
for all }\quad (w, z)\in \left(\bigcup_{j=1}^N \left( \left( \overline{\Om} \cap \overline{U_j} \right) \times (\overline{\Om}\setminus U_j)\right)\bigcup \left(\overline{U_0} \times \overline{\Om}\right)\right).
$$
Using this and Proposition A for $U_j$ and $(a, b, s) = (2, \beta, \alpha)$ with $2\alpha + \beta < 2$, we get that  for every $z \in \Om$,
$$
\int_{\Om \cap U_j} |K_{\Om}(w,z)|^2 K_{\Om}(w,w)^{-\alpha} d_{\Om}(w)^{\beta} dV(w) \lesssim
C, \text{ if } z \in \Om \setminus U_j
$$
and
$$
\int_{\Om \cap U_j} |K_{\Om}(w,z)|^2 K_{\Om}(w,w)^{-\alpha} d_{\Om}(w)^{\beta} dV(w) \lesssim
C_j K_{\Om}(z,z)^{1 - \alpha} d_{\Om}(z)^{\beta}, \text{ if } z \in \Om \cap U_j.
$$
Hence, for every $z \in \Om$,
\begin{align*}
 \int_{\Om} |K_{\Om}(w,z)|^2 K_{\Om}(w,w)^{-\alpha} d_{\Om}(w)^{\beta} dV(w) \leq &  \int_{U_0} |K_{\Om}(w,z)|^2  K_{\Om}(w,w)^{-\alpha} d_{\Om}(w)^{\beta} dV(w) \\
  + & \sum_{j = 1}^N \int_{\Om \cap U_j} |K_{\Om}(w,z)|^2  K_{\Om}(w,w)^{-\alpha} d_{\Om}(w)^{\beta} dV(w) \\
 \lesssim & \; C^2 + \sum_{j = 1}^N \max \left\{C^2,  C_j K_{\Om}(z,z)^{1-\alpha}d_{\Om}(z)^{\beta} \right\}. 
\end{align*}
Moreover, since $2\alpha + \beta < 2$ and $d_{\Om}(z)^{-2} \leq K_{\Om}(z,z), z \in \Om$,
$$
K_{\Om}(z,z)^{1-\alpha}d_{\Om}(z)^{\beta} \geq d_{\Om}(z)^{2\alpha + \beta - 2} \to \infty \text{ as } z \to \partial \Om.
$$
From this and the above inequality the desired estimate follows.
\end{proof}

\bigskip

\textbf{Acknowledgement.} This paper has been carried out during the authors' stay at the Vietnam Institute for Advanced Study in Mathematics. They would like to thank the institution for hospitality and support

 \vspace{10pt}
 

\begin{thebibliography}{10}
	\bibitem{ARS12}
	M.~Abate, J.~Raissy, and A.~Saracco.
	\newblock Toeplitz operators and {C}arleson measures in strongly pseudoconvex domains.
	\newblock {\em J. Funct. Anal.}, 263(11):3449--3491, 2012.
	
	  \bibitem{APP15}
	H.~Arroussi, i.~Park and J.~Pau.
	\newblock Schatten class Toeplitz operators acting on large weighted Bergman spaces.
	\newblock {\em Studia Math.}, 229:203--221, 2015.
		
	
	
	\bibitem{Cat89}
	D.W. Catlin.
	\newblock Estimates of invariant metrics on pseudoconvex domains of dimension
	two.
	\newblock {\em Math. Z.}, 200(3):429--466, 1989.
	
	
	\bibitem{Cho96}
	S.~Cho.
	\newblock Estimates of the {B}ergman kernel function on certain pseudoconvex
	domains in {${\bf C}^n$}.
	\newblock {\em Math. Z.}, 222(2):329--339, 1996.
	
	\bibitem{Cho02}
	S.~Cho.
	\newblock Estimates of the {B}ergman kernel function on pseudoconvex domains
	with comparable {L}evi form.
	\newblock {\em J. Korean Math. Soc.}, 39(3):425--437, 2002.
	
	\bibitem{Cuc17}
	Z.~\u{C}u\u{c}kovi\'c.
	\newblock Estimates of the {$L^p$} norms of the {B}ergman projection on
	strongly pseudoconvex domains.
	\newblock {\em Integral Equations Operator Theory}, 88(3):331--338, 2017.
	
	\bibitem{CuMc06}
	Z.~\u{C}u\u{c}kovi\'c and J.D. McNeal.
	\newblock Special {T}oeplitz operators on strongly pseudoconvex domains.
	\newblock {\em Rev. Mat. Iberoam.}, 22(3):851--866, 2006.
	
	\bibitem{CS13}
	Z.~\u{C}u\u{c}kovi\'c and S.~\c{S}ahuto\u{g}lu.
	\newblock Axler-Zheng type theorem on a class of domains in $\mathbb C^n$.
	\newblock {\em Integr. Equ. Oper. Theory}, 77:397--405, 2013.
	
	\bibitem{Fef74}
	C.~Fefferman.
	\newblock The {B}ergman kernel and biholomorphic mappings of pseudoconvex
	domains.
	\newblock {\em Invent. Math.}, 26:1--65, 1974.
	
	
	
	
	\bibitem{KLT18} T.~V.~Khanh, J.~Liu, P.~T.~Thuc. 
	\newblock{Bergman-Toeplitz operators on weakly pseudoconvex domains.} 
	\newblock{\em Math. Z.}, 291(2):591--607, 2019.	
	
	\bibitem{McN91}
	J.D. McNeal.
	\newblock Local geometry of decoupled pseudoconvex domains.
	\newblock In {\em Complex analysis ({W}uppertal, 1991)}, Aspects Math., E17,
	pages 223--230. Friedr. Vieweg, Braunschweig, 1991.
	
	
	\bibitem{McN94}
	J.D.\ McNeal.
	\newblock Estimates on the {B}ergman kernels of convex domains.
	\newblock {\em Adv.\ Math.}, 109:108--139, 1994.


	\bibitem{McNSt94}
	J.D.\ McNeal and E.M.\ Stein.
	\newblock Mapping properties of the {B}ergman projection on convex domains of
	finite type.
	\newblock {\em Duke Math.\ J.}, 73:177--199, 1994.
	
	\bibitem{MSW13}
	M.~Mitkovski, M.~Su\'arez and B.~D.~Wick.
	\newblock The essential norm of operators on $A^p_{\alpha}(\mathbb B_n)$.
	\newblock {\em Integr. Equ. Oper. Theory}, 75(2):197–233, 2013.
	
	\bibitem{NaRoStWa89}
	A.\ Nagel, J.-P.\ Rosay, E.M.\ Stein, and S.\ Wainger.
	\newblock Estimates for the {B}ergman and {S}zeg{\"o} kernels in ${{\mathbb
			C}}^2$.
	\newblock {\em Ann.\ of Math.}, 129:113--149, 1989.
	
  \bibitem{P14}
	J.~Pau.
	\newblock A remark on Schatten class Toeplitz operators on Bergman spaces.
	\newblock {\em Proc. Amer. Math. Soc.}, 142:2763--2768, 2014.
	
	\bibitem{PhSt77}
	D.H.\ Phong and E.M.\ Stein.
	\newblock Estimates for the {B}ergman and {S}zeg{\"o} projections on strongly
	pseudo-convex domains.
	\newblock {\em Duke Math.\ J.}, 44(3):695--704, 1977.
	
	\bibitem{S98}
	K.~Stroethoff.
	\newblock Compact Toeplitz operators on Bergman spaces.
	\newblock {\em Math. Proc. Camb. Phil. Soc.}, 124(1):151–160, 1998.
		
		
	\bibitem{S07}
	D.~Su\'arez.
	\newblock The essential norm of operators in the Toeplitz algebra on $A^p(\mathbb B_n)$.
	\newblock {\em Indiana Univ. Math. J.}, 56(5):2185–2232, 2007.
	
	
	
	
	
	
	\bibitem{TV10}
	J.~Taskinen and J.~Virtanen.
	\newblock Toeplitz operators on Bergman spaces with locally integrable symbols.
	\newblock {\em Rev. Mat. Iberoamericana}, 26(2):693--706, 2010.
	
	\bibitem{Zhu06}
	K.~Zhu.
	\newblock A sharp norm estimate of the {B}ergman projection on {$L^p$} spaces.
	\newblock In {\em Bergman spaces and related topics in complex analysis},
	volume 404 of {\em Contemp. Math.}, pages 199--205. Amer. Math. Soc.,
	Providence, RI, 2006.

	\bibitem{Z07}
	K.~Zhu.
	\newblock Schatten class Toeplitz operators on weighted Bergman spaces of the unit ball.
	\newblock {\em New York J. Math.}, 13:299--316, 2007.
	
  \bibitem{Zhu07} K.~Zhu.
  \newblock Operator Theory in Function Spaces.
  \newblock Second Edition, Math. Surveys and Monographs 138, American Mathematical Society: Providence, Rhode Island, 2007.	
	
\end{thebibliography}

\vspace{30pt}

\end{document}